\documentclass[11pt,reqno]{amsart}
\usepackage[utf8]{inputenc} 

\usepackage[margin=1in]{geometry} 

\usepackage{graphicx} 
\usepackage{float} 

 \usepackage[parfill]{parskip} 
 
\usepackage{booktabs} 
\usepackage{array} 
\usepackage{paralist} 
\usepackage{verbatim} 
\usepackage{subfig} 
\usepackage{mathrsfs}
\usepackage{amssymb}
\usepackage{xcolor}
\usepackage{amsthm}
\usepackage{amsmath,amsfonts,amssymb,esint,hyperref}
\usepackage[noabbrev, capitalize]{cleveref}

\usepackage{autonum}

\usepackage{graphics,color}
\usepackage{enumerate, enumitem}
\usepackage{mathtools,centernot}
\usepackage{cases}
\usepackage{amsrefs}
\usepackage{bbm}
\usepackage{xfrac}
\usepackage{anyfontsize}

\usepackage{bookmark}

\newtheorem{theorem}{Theorem}[section]
\newtheorem{lemma}[theorem]{Lemma}

\newtheorem{corollary}[theorem]{Corollary}
\newtheorem{definition}[theorem]{Definition}
\newtheorem{proposition}[theorem]{Proposition}
\newtheorem{remark}[theorem]{Remark}

\numberwithin{equation}{section} 

\newcommand{\norm}[1]{\left\|#1\right\|}
\newcommand{\abs}[1]{\left|#1\right|}

\newcommand{\T}{\ensuremath{\mathbb{T}}}
\newcommand*{\R}{\ensuremath{\mathbb{R}}}

\newcommand*{\N}{\ensuremath{\mathbb{N}}}

\newcommand{\eps}{\varepsilon}

\newcommand{\quotes}[1]{``#1''}

\renewcommand{\MR}[1]{} 

\usepackage{color, graphicx}
\usepackage{mathrsfs, dsfont}

\usepackage[]{hyperref}
\hypersetup{
    colorlinks=true,       
    linkcolor=red,          
    citecolor=blue,        
    filecolor=red,      
    urlcolor=cyan           
}

\def\div{\mathop{\rm div}\nolimits}    
\def\curl{\mathop{\rm curl}\nolimits}    
 
\def\spt{\mathop{\rm Spt}\nolimits} 
 
\def\Lip{\mathop{\rm Lip}\nolimits}

\newcommand{\be}{\begin{equation}}
\newcommand{\ee}{\end{equation}}

\title{Dissipation concentration in two-dimensional fluids}

\author[L. De Rosa]{Luigi De Rosa}
\address[L. De Rosa]{Gran Sasso Science Institute, viale Francesco Crispi, 7, 67100 L’Aquila, Italy}
\email{luigi.derosa@gssi.it}

\author[J. Park]{Jaemin Park}
\address[J. Park ]{Department of Mathematics,
Yonsei university
50 Yonsei-Ro, Seodaemun-Gu, 03722 Seoul, South Korea}
\email{jpark776@yonsei.ac.kr}

\date{\today}

\thanks{\textit{Acknowledgments} JP was partially supported by SNSF Ambizione
fellowship project PZ00P2-216083, the Yonsei University Research Fund of 2024-22-0500, and the
POSCO Science Fellowship of POSCO TJ Park Foundation.}

\begin{document}

\begin{abstract}
We study the dissipation measure arising in the inviscid limit of two-dimensional incompressible fluids. It is proved that the dissipation is Lebesgue in time and, for almost every time, it is absolutely continuous with respect to the defect measure of strong compactness of the solutions. When the initial vorticity is a measure, the dissipation is proved to be absolutely continuous with respect to a ``quadratic'' space-time vorticity measure. This results into the trivial measure if the initial vorticity has singular part of distinguished sign, or a spatially purely atomic measure if wild oscillations in time are ruled out. In fact,  the dynamics at the Batchelor--Kraichnan dissipative scale is the only relevant one, in turn offering new criteria for anomalous dissipation. We provide  kinematic examples highlighting the strengths and the limitations of our approach. Quantitative rates, dissipation life-span and steady fluids are also investigated. 
\end{abstract}

\maketitle

\section{Introduction}
We consider the two-dimensional Navier--Stokes equations
\begin{equation}\label{NS} \tag{NS}
\left\{\begin{array}{ll}
\partial_t u^\nu +\div (u^\nu \otimes u^\nu) +\nabla p^\nu=\nu \Delta u^\nu \\
\div u^\nu=0 \\
u^\nu(\cdot, 0)=u^\nu_0
\end{array}\right.
\end{equation}
on $\T^2 \times [0,T)$. We are interested in the behavior as $\nu\rightarrow 0$, where phenomena related to turbulence happen. For any $\nu>0$ and any $u^\nu_0 \in L^2(\T^2)$,  global weak solutions $u^\nu \in L^\infty ([0,T]; L^2(\T^2))\cap L^2([0,T]; H^1(\T^2))$ are known to exist since the seminal works of Leray \cite{L34} and Hopf \cite{Hopf51}. The pressure can be then recovered a posteriori as the unique zero-average solution to
$$
-\Delta p^\nu =\div \div (u^\nu \otimes u^\nu).
$$
In two space dimensions, they are unique \cites{RR,BV22}, they instantaneously become smooth, and they satisfy the energy equality
\begin{equation}
    \label{NS en bal}
    \frac{1}{2}\| u^\nu (t)\|_{L^2_x}^2 + \nu \int_0^t \|\nabla u^\nu (s)\|^2_{L^2_x}\,ds= \frac{1}{2}\| u^\nu_0\|_{L^2_x}^2\qquad \forall t\in [0,T].
\end{equation}
 By standard weak compactness arguments, we will often pass to subsequences without specifying it. A direct consequence of \eqref{NS en bal} is that a sequence of $L^2(\T^2)$ bounded initial data results into a sequence of solutions $\{u^\nu\}_{\nu}$ bounded in $L^\infty ([0,T]; L^2(\T^2))$, with dissipation $\{\nu |\nabla u^\nu|^2\}_{\nu}$ bounded in $L^1(\T^2\times [0,T])$. In particular, if $u^\nu  \overset{*}{\rightharpoonup} u$ in $L^\infty ([0,T]; L^2(\T^2))$, we deduce that $\{|u^\nu -u|^2\}_{\nu}$ is bounded in $L^\infty ([0,T]; L^1(\T^2))$. We can thus define the \quotes{dissipation measure} and the \quotes{defect measure}, denoted by $D$ and $\Lambda$ respectively, as
\begin{align} 
    \nu |\nabla u^\nu|^2 &\overset{*}{\rightharpoonup} D \qquad  \text{in } \mathcal M (\T^2\times [0,T]),\\
    |u^\nu -u|^2&\overset{*}{\rightharpoonup} \Lambda \qquad \text{in } L^\infty ([0,T]; \mathcal M(\T^2)).
    \end{align}
 Let us denote the vorticity by $\omega^\nu:= \curl u^\nu$ and $\omega^\nu_0:=\curl u^\nu_0$. Since $\|\omega^\nu (t)\|_{L^2_x}=\|\nabla u^\nu (t)\|_{L^2_x}$, the sequence $\{\nu|\omega^\nu|^2\}_\nu$ generates a measure equivalent to $D$ (see Proposition \ref{P: D and tilde D are equiv}). Moreover, because of the transport structure of the vorticity in two dimensions, the sequence $\{\omega^\nu\}_{\nu}$ stays bounded in $L^\infty ([0,T]; L^1(\T^2))$
 as soon as $\{\omega^\nu_0\}_{\nu}$ is bounded in $\mathcal M (\T^2)$. This allows to define the \quotes{vorticity measure}, denoted by $\Omega$, as 
 \begin{equation}
     |\omega^\nu |\overset{*}{\rightharpoonup} \Omega\qquad \text{in } L^\infty ([0,T]; \mathcal M (\T^2)).
     \end{equation}

Let us remark that none of the above measures is uniquely determined as different subsequences might lead to different limits. 
These three fundamental objects have been playing a major role towards the understanding of the intricate dynamics of incompressible fluids at high Reynolds numbers. Getting a non-trivial $D$ in the inviscid limit goes under the name ``anomalous dissipation", a phenomenon that relates to the presumed ``universality" of turbulence since the foundational works of Kolmogorov \cite{K41} and Onsager \cite{O49}. The measure $\Lambda$, or a \quotes{reduced} version of it \cite{DM88}, quantifies the lack of strong compactness and it is related to the inviscid limit problem as settled in the seminal papers by DiPerna and Majda \cites{diperna1987concentrations, DM87,DM88}, while the vorticity measure $\Omega$ relates to a remarkable concentration compactness argument as first noticed by Delort \cite{delort1991existence}. The main objective of our paper is to study the relation, if any, between these three objects, going beyond what is expected to happen in the three dimensional setting. As it turns out, the approach we propose generalizes all the results from \cites{LMP21,DRP24,CLLS16,ELL_tocome}. However, none of our arguments makes use of \quotes{Gagliardo--Nirenberg \& super-quadratic Gr\"onwall} (or improved versions of it \cite{ELL_tocome}), which was the common strategy in \cites{LMP21,CLLS16,ELL_tocome}. Since several directions are explored, we group them in different subsections. We emphasize that none of the results requires the weak limit to be a distributional solution to the incompressible Euler equations. This goes beyond the previous approaches enlarging the applicability range to the, nowadays not yet excluded, scenario in which oscillations and/or concentrations persist in the inviscid limit.  

\subsection{The measures of dissipation, defect and vorticity (Section~\ref{Measure_comparison})} In this subsection, we investigate relations between the measures $\Lambda$, $D$ and $\Omega$. All the measures considered in this paper will be finite non-negative Borel measures. Given $p\in [1,\infty]$, we recall that $\mu \in L^p ([0,T]; \mathcal M (\T^2))$ if $\mu =\mu_t \otimes dt$ for a weakly measurable\footnote{Weakly measurable means that the map $t\mapsto \langle \mu_t ,\varphi\rangle$ is measurable for any $\varphi \in C^0 (\T^2)$.} map $t\mapsto \mu_t \in \mathcal M (\T^2)$ such that $\mu_t(\T^2)\in L^p([0,T])$. If $\mu,\lambda$ are two measures, we say that $\mu$ is \quotes{absolutely continuous} with respect to $\lambda$, written as $\mu\ll\lambda$, if $\mu(A)=0$ for any measurable set $A$ such that $\lambda (A)=0$. 

We can now state our first theorem. {Let us emphasize that Theorem \ref{T: general leray} below is a fully unconditional statement. As such, it applies to all weak limits of Leray--Hopf solutions. We are not aware of any other unconditional result in this context.

\begin{theorem}\label{T: general leray}
Let  $\{u_0^\nu\}_{\nu}\subset L^2(\T^2)$ be a strongly compact sequence of divergence-free vector fields and let $\{u^\nu\}_{\nu}$ be the corresponding sequence of Leray--Hopf solutions to \eqref{NS}. Assume that 
    \begin{equation}
    \nu|\nabla u^{\nu}|^2\overset{*}{\rightharpoonup} D  \qquad \text{in } \mathcal{M}(\mathbb{T}^2\times [0,T]). \label{assumption2_l}
    \end{equation}
    Then $D\in L^1([0,T];\mathcal{M}(\mathbb{T}^2))$. In addition, assume that $u^\nu  \overset{*}{\rightharpoonup} u$  and $|u^\nu -u|^2 \overset{*}{\rightharpoonup} \Lambda$, respectively in  $L^\infty([0,T];L^2(\T^2))$ and in $L^\infty([0,T];\mathcal M(\T^2))$. Then $D_t \ll \Lambda_t$ for a.e. $t\in [0,T]$.
\end{theorem}
In fact, we will prove that $D\in L^\infty_{\rm{loc}}((0,T];\mathcal{M} (\T^2))$ and then the strong $L^2 (\T^2)$ compactness of the initial data is used to rule out atomic concentrations at the initial time (see Proposition \ref{P:dissipation short times quantitative}). This is in fact the only use we make of the initial compactness,  while all the other properties proved for $D_t$ would still be true even without that assumption (see Remark \ref{R:no need of compactness initial}). The property $D_t \ll \Lambda_t$ generalizes, by making it completely local, the main result of \cite{LMP21} proving that strong $L^2 (\T^2\times [0,T])$ compactness rules out anomalous dissipation. In fact, the proof gives a quantitative relation between $D$ and $\Lambda$ for all positive times (see Remark \ref{R:D vs Lambda quantitative}), which is strictly stronger than absolute continuity.

\begin{remark}
    We emphasize that all the results proved in the current paper do not follow the classical approach \cites{CET94,DR00,DRIS24,DDI24,CCFS08,DDII25} in which properties of $D$ are  deduced  by looking at the local energy balance 
    \begin{equation}\label{en bal local}
    (\partial_t - \nu \Delta)\frac{\abs{u^\nu}^2}{2} +\div  \left(\left(\frac{\abs{u^\nu}^2}{2}+p^\nu\right)u^\nu\right) =- \nu \abs{\nabla u^\nu}^2.
        \end{equation}
    Proving any property on $D$ from \eqref{en bal local} would at least require a control of $u^\nu$ in $L^3(\T^2\times[0,T])$, thus out of our setting. It is then necessary to develop a strategy which can capture properties of $D$ by never looking at \eqref{en bal local} locally. This seems to be possible only in two dimensions. 
\end{remark}

Whenever $\{\omega^\nu_0\}_\nu$ is bounded in $\mathcal M (\T^2)$, also the measure $\Omega$ comes into play, imposing stronger constraints on the dissipation.
\begin{theorem}
    \label{T: measure vort}
Let  $\{u_0^\nu\}_{\nu}\subset L^2(\T^2)$ be a strongly compact sequence of divergence-free vector fields such that $\{\omega^\nu_0\}_\nu \subset \mathcal M(\T^2)$ is  bounded. Let $\{u^\nu\}_{\nu}$ be the corresponding sequence of Leray--Hopf solutions to \eqref{NS} and define 
\begin{equation}\label{modified vort}
\hat \Omega^\nu(x,t):= |\omega^{\nu}(x,t)| \int_{B_{\sqrt{\nu}}(x)}  |\omega^{\nu}(y,t)| \,dy.
\end{equation}
Assume 
\begin{itemize}
    \item[(i)]  $\nu|\nabla u^{\nu}|^2\overset{*}{\rightharpoonup} D$ in $\mathcal{M}(\mathbb{T}^2\times [0,T])$;
    \item[(ii)] $u^\nu  \overset{*}{\rightharpoonup} u$ in  $L^\infty([0,T];L^2(\T^2))$ and   $|u^\nu -u|^2 \overset{*}{\rightharpoonup} \Lambda$ in $L^\infty([0,T];\mathcal M(\T^2))$;
        \item[(iii)] $|\omega^\nu|\overset{*}{\rightharpoonup} \Omega$ in $L^\infty([0,T];\mathcal M(\T^2))$; 
    \item[(iv)] $\hat \Omega^\nu \overset{*}{\rightharpoonup} \hat \Omega$ in $L^\infty([0,T];\mathcal M(\T^2))$.
\end{itemize}
Then $D\in L^1([0,T];\mathcal{M}(\mathbb{T}^2))$, $D_t\ll \Lambda_t$, $D_t\ll  \hat \Omega_t$ and $D_t\ll \Omega_t$ for a.e. $t\in [0,T]$. 
\end{theorem}

Note that $\{ \hat \Omega^\nu\}_{\nu}\subset L^\infty ([0,T]; L^1(\T^2))$ is bounded and the assumption $(iv)$ is always achieved by compactness. Theorem \ref{T: measure vort} generalizes our previous result \cite{DRP24}, which was itself generalizing \cite{CLLS16} where the very first Onsager supercritical energy conservation condition was obtained for $L^p (\T^2)$ initial vorticity, $p>1$. Indeed, when the initial vorticity has positive\footnote{A singular part with distinguished sign suffices.} singular part it can be proved that $ \hat \Omega=0$. Moreover, when $|\omega^\nu|\otimes |\omega^\nu|$ converges to a product measure, $D$ is spatially purely atomic. We collect these considerations in the following corollary. 

We recall that, given a measure $\mu$ and a Borel set $A$, the symbol $\mu\llcorner A$ denotes the restriction of $\mu$ to $A$, that is $\mu\llcorner A (B):=\mu(A\cap B)$ for all Borel sets $B$. Consequently, we say that $\mu$ is concentrated on $A$ if $\mu=\mu\llcorner A$, or equivalently $\mu(A^c)=0$.

\begin{corollary}\label{C: pos vort and atomic}
Under all the assumptions of Theorem \ref{T: measure vort} the following hold. 
\begin{itemize}
    \item[(a)] If $\omega_0^\nu=f^\nu_0 +\mu_0^\nu$ with $\{f^\nu_0\}_{\nu}\subset L^1(\T^2)$ weakly compact and $\mu_0^\nu \geq 0$, then $ \hat \Omega=0$ and consequently $D=0$.
    \item[(b)] Assume  that $|\omega^\nu|\otimes |\omega^\nu| \overset{*}{\rightharpoonup} \Gamma$ in $L^\infty([0,T]; \mathcal M(\T^2 \times \T^2))$ and there exists $\gamma\in L^\infty([0,T]; \mathcal M(\T^2 ))$ such that $\Gamma_t=\gamma_t\otimes \gamma_t$ for a.e. $t\in [0,T]$. Denoting by $\mathscr L_t$ and $\mathscr O_t$ the sets of atoms of $\Lambda_t$ and $\Omega_t$ respectively, i.e. 
    $$
    \mathscr L_t:= \left\{ x\in\mathbb{T}^2\,:\, \Lambda_t(\left\{ x\right\})>0\right\}\qquad \text{ and } \qquad \mathscr O_t:= \left\{ x\in\mathbb{T}^2\,:\, \Omega_t(\left\{ x\right\})>0\right\} ,
    $$
    we have that $D_t$ is concentrated on $\mathscr L_t\cap \mathscr O_t$, i.e. $D_t =D_t \llcorner \left(\mathscr L_t \cap \mathscr O_t\right)$ for a.e. $t\in [0,T]$.
\end{itemize}
\end{corollary}

A practical assumption which guarantees that $\Gamma$ is a product measure in space is when $|\omega^\nu_t|\overset{*}{\rightharpoonup} \Omega_t$ for a.e. $t$. In this case $\Gamma_t = \Omega_t\otimes \Omega_t$. However, in view of wild oscillations in time, this might fail in general (see Remark \ref{R:schocet examp}). A slightly weaker assumption on $\Gamma_t$ is discussed in Remark \ref{R:relax product assumption}.

The argument used to prove $(b)$ in Corollary \ref{C: pos vort and atomic} is sharp (see Remark \ref{R:sharp atomic dynamics}). Being $\Lambda$ and $\Omega$ finite measures, the sets $\mathscr L_t$ and $\mathscr O_t$ are at most countable for a.e. $t$. When the vorticity is a measure, the fact that $D_t$ is purely atomic aligns with the well known concentration compactness principle by Lions \cites{PLL1,PLL2}. However, this is in some sense quite surprising. Let us explain why. One of the easiest applications of the Lions argument is the study of compactness in the Sobolev embedding $W^{1,1} (\T^2) \subset L^2 (\T^2)$ (see for instance \cite{Struwe}*{Section 4.8} and \cite{DT23} for recent generalizations). In this setting, the concentration compactness principle shows that the loss of $L^2 (\T^2)$ compactness is fully characterized by a purely atomic measure concentrated on the set of atoms appearing in the absolute value of the gradient. However, as we shall show in Proposition \ref{P: no weak lions}, the failure of the Calder\'on--Zygmund estimate in $L^1 (\T^2)$ allows the defect measure $\Lambda$ to diffuse even if the vorticity is a measure.  Arguing this way, the naive 
interpretation of \eqref{NS en bal} as $\nu|\nabla u^\nu|^2\sim |u^\nu|^2$, would suggest that $D$ should diffuse as well, as opposed to what it is proved in part $(b)$ of Corollary \ref{C: pos vort and atomic}. Of course, this reasoning is \quotes{modulo time oscillations}, which leads us to also consider the steady case where this is proved in full generality (see Theorem \ref{T: measure vort steady} below). In particular, although the end point failure of Calder\'on--Zygmund, a measure vorticity always constraints the  dissipation to fully concentrate in space, and wild oscillations in time are the only true obstacle.

\begin{remark}
The fact that $D_t$ is concentrated on $\mathscr{ L}_t\cap \mathscr O_t$ shows that, in order to observe a non-trivial dissipation, spatial atomic concentrations must simultaneously happen for both $\Lambda_t$ and $\Omega_t$, at the same time $t$ and at the same point $x$. As we shall prove in Proposition \ref{P: atoms in lambda vs omega}, the concentration of any of the two measures might, in principle, happen independently on the other. In other words, it might be possible that $\mathscr{L}_t\cap \mathscr O_t=\emptyset$ for a.e. $t$, even if none of the two is empty.
\end{remark}

\subsection{The dissipative scale and anomalous dissipation criteria (Section~\ref{S:K41 scale})}
The main objective of this section is to show that strong compactness and vorticity concentration at the dissipative scale fully characterize anomalous dissipation in two dimensions.

A consequence of Theorem \ref{T: general leray} is that strong compactness of  $\{u^\nu\}_{\nu}\subset L^2 (\T^2\times [0,T])$ implies $D=0$. As previously proved in \cite{LMP21}, strong $L^2 (\T^2\times [0,T])$ compactness is in fact equivalent to energy conservation of the inviscid limit. In the direction of quantifying the relevant scales contributing to the energy dissipation, we are able to show that the ones above the \quotes{dissipative scale} do not matter at all. In the two-dimensional setting this corresponds to consider length scales\footnote{In the physics literature, the dissipative scale in two dimensions is known as the Batchelor--Kraichnan scale \cites{Batch69,Kraich67} and it relates to enstrophy dissipation.} $\sim \sqrt{\nu}$. It follows that the \quotes{inertial range} is always deprived of energetic content independently on any uniform (in viscosity) regularity retained at these scales.

Given $\mathcal S^\nu_2(y,t):=\norm{u^\nu (\cdot + y, t) - u^\nu(\cdot, t)}_{L^2_x}^2$, for any $\ell>0$ we define 
\begin{align}
  S^\nu_2(\ell):= \int_0^T \fint_{B_{\ell}(0)}\mathcal S^\nu_2(y,t) \, dydt. \label{ell compact uniform}
\end{align}
This object relate to \quotes{absolute structure functions} of second order\footnote{{It is a way to measure Besov regularity in the space variable (see Remark \ref{R:besov equivalence}).}}, which play a major role in the context of turbulent fluids \cite{Frisch95}.  Being of second order, it coincides with the longitudinal one \cite{Driv22} for solutions to \eqref{NS}. Consider now a sequence of positive numbers $\{\ell_\nu\}_\nu$. In light of the classical Fr\'echet--Kolmogorov compactness criterion\footnote{The condition $\lim_{\ell\rightarrow 0}\sup_{\nu>0}   S^\nu_2(\ell)=0$ becomes truly equivalent to strong compactness in space-time. See for instance \cite{LMP21}*{Theorem 2.11}.}, we shall refer to $  S^\nu_2(\ell_\nu)\rightarrow 0$ as \quotes{compactness at scale $\ell_\nu$}.  
\begin{theorem}
    \label{T: K41 scale intro}
    Let  $\{u_0^\nu\}_{\nu}\subset L^2(\T^2)$ be a strongly compact sequence of divergence-free vector fields and let $\{u^\nu\}_{\nu}$ be the corresponding sequence of Leray--Hopf solutions to \eqref{NS}. Then
    \begin{equation}
        \label{k41 comp implies no AD}
        \lim_{\nu\rightarrow 0}  S^\nu_2(\sqrt{\nu})=0  \qquad \Longleftrightarrow \qquad \lim_{\nu\rightarrow 0} \nu \int_0^T \|\nabla u^\nu (t)\|^2_{L^2_x}\,dt=0.
    \end{equation}
\end{theorem}

As well as \cite{LMP21} shows that the $L^2 (\T^2\times [0,T])$ compactness is equivalent to energy conservation of the inviscid limit, Theorem \ref{T: K41 scale intro}  shows that $D=0$ is equivalent to the  compactness at the dissipative length scale\footnote{In fact, it is more likely that $\sqrt{\nu}$ appears for scaling reasons unrelated to the usual considerations used to identify the dissipative range. Indeed, the latter relates to the enstrophy anomaly.}. As a consequence, we also capture the sharp length scales on quadratic structure functions decay that has been considered in \cite{LMP21}. Further comments as well as more general and quantitative versions of Theorem \ref{T: K41 scale intro} will be given in Section \ref{S:K41 scale}.

As it was maybe already apparent from the definition of $ \hat \Omega^\nu$ in \eqref{modified vort}, the same phenomenon happens at the \quotes{concentration level} as soon as the initial vorticity is a measure. In order to state the next theorem, let us define the (global) concentrated 
versions of $\Lambda$ and $\Omega$ at scale $\ell$ as
\begin{align}
   \Lambda^{\nu}_{\rm con} (\ell) &:= \int_0^T \left(\sup_{x\in \T^2}\int_{B_{\ell}(x)} | u^\nu(y,t)-u (y,t)|^2\,dy \right)^\frac12 dt ,\label{lambda concentr k41 scale}\\
    \Omega^{\nu}_{\rm con} (\ell)&:= \int_0^T \left(\sup_{x\in \T^2}\int_{B_{\ell}(x)} |\omega^\nu (y,t)|\,dy \right)dt. \label{lambda concentr k41 scale}
\end{align}
In defining $\Lambda^{\nu}_{\rm con} (\ell)$ we are implicitly assuming that  $u^\nu  \overset{*}{\rightharpoonup} u$ in  $L^\infty ([0,T]; L^2(\T^2))$. Note that, since  $u\in L^\infty ([0,T]; L^2(\T^2))$,  the absolute continuity of the Lebesgue integral, together with the dominated convergence theorem applied in the time variable, yields to
\begin{equation}\label{lambda con equivalent definition}
\lim_{\nu\rightarrow 0}\Lambda^{\nu}_{\rm con} (\ell_\nu)=0 \qquad \Longleftrightarrow\qquad \lim_{\nu\rightarrow 0}\int_0^T \left(\sup_{x\in \T^2}\int_{B_{\ell_\nu}(x)} | u^\nu(y,t)|^2\,dy \right)^\frac12 dt=0
\end{equation}
for any length scales such that $\ell_\nu\rightarrow 0$ as $\nu\to 0$.
\begin{theorem}\label{T: k41 scale concentration}
Let  $\{u_0^\nu\}_{\nu}\subset L^2(\T^2)$ be a strongly compact sequence of divergence-free vector fields such that $\{\omega^\nu_0\}_\nu \subset \mathcal M(\T^2)$ is  bounded. Let $\{u^\nu\}_{\nu}$ be the corresponding sequence of Leray--Hopf solutions to \eqref{NS} and assume that  $u^\nu  \overset{*}{\rightharpoonup} u$ in  $L^\infty([0,T];L^2(\T^2))$. Then
\begin{equation}
    \label{condition on veloc L2}
    \lim_{\nu\rightarrow 0} \Lambda^{\nu}_{\rm con} (\sqrt{\nu})=0\qquad \Longrightarrow \qquad \lim_{\nu\rightarrow 0}\nu \int_0^T \|\nabla u^\nu (t)\|^2_{L^2_x}\,dt=0
\end{equation}
and 
\begin{equation}
    \label{condition on vort L1}
     \lim_{\nu\rightarrow 0} \Omega^{\nu}_{\rm con} (\sqrt{\nu})=0\qquad \Longleftrightarrow \qquad \lim_{\nu\rightarrow 0}\nu \int_0^T \|\nabla u^\nu (t)\|^2_{L^2_x}\,dt=0.
\end{equation}
In particular
\begin{equation}\label{lambda vs omega implication for NS}
 \lim_{\nu\rightarrow 0} \Lambda^{\nu}_{\rm con} (\sqrt{\nu})=0\qquad \Longrightarrow \qquad \lim_{\nu\rightarrow 0} \Omega^{\nu}_{\rm con} (\sqrt{\nu})=0.
\end{equation}
\end{theorem}

It follows that, when the initial vorticity is finite measure, vorticity concentration at the dissipative scale gives another criterion for anomalous dissipation. Although atomic concentrations in $\Lambda$ might occur independently on the ones in $\Omega$ for general sequences of divergence-free vector fields (see Proposition \ref{P: atoms in lambda vs omega}), restricting to solutions to \eqref{NS} makes the one-sided implication \eqref{lambda vs omega implication for NS} true at the dissipative length scale. 

Although considering the quantity $\Lambda_{\rm con}^\nu(\ell)$ is quite natural in the spirit of this paper, a little adjustment allows to substitute \eqref{condition on veloc L2} with a full equivalence. To do that, we set
\begin{equation}
    \label{Q nu definition}
    Q^\nu_{\rm con}(\ell):=\int_0^T \left(\sup_{x\in \T^2}\int_{B_{\ell}(x)} \left| u^\nu(y,t)-\fint_{B_{\ell}(x)} u^\nu (z,t) \,dz\right|^2\,dy \right)^\frac12 dt.
\end{equation}
Note that $\Lambda^\nu_{\rm con}$ controls $Q^\nu_{\rm con}$. Indeed, by \eqref{lambda con equivalent definition} we have
$$
\lim_{\nu\rightarrow 0}\Lambda^{\nu}_{\rm con} (\ell_\nu)=0 \qquad \Longrightarrow\qquad \lim_{\nu\rightarrow 0}Q^\nu_{\rm con}(\ell_\nu)=0
$$
as soon as $\ell_\nu\rightarrow 0$.
\begin{theorem}
    \label{T:new full charact at dissipative scale}
    Let  $\{u_0^\nu\}_{\nu}\subset L^2(\T^2)$ be a strongly compact sequence of divergence-free vector fields such that $\{\omega^\nu_0\}_\nu \subset \mathcal M(\T^2)$ is  bounded. Let $\{u^\nu\}_{\nu}$ be the corresponding sequence of Leray--Hopf solutions to \eqref{NS}. Then 
     \begin{equation}\label{new full equivalence equation}
   \lim_{\nu\rightarrow 0} Q^\nu_{\rm con}(\sqrt{\nu}) =0  \quad \Longleftrightarrow \quad \lim_{\nu\rightarrow 0} \nu \int_0^T \|\nabla u^\nu (t)\|^2_{L^2_x}\,dt=0\quad \Longleftrightarrow \quad \lim_{\nu\rightarrow 0} \Omega^\nu_{\rm con}(\sqrt{\nu}) =0.
      \end{equation}
\end{theorem}

Refined local versions of all the theorems above can be also obtained (see  Remark \ref{R:K41 compactness and no AD localized} and Remark \ref{R:K41 concentration and no AD localized}). Moreover, considering scales that are ``asymptotically'' at most (or at least) $\ell_\nu\sim \sqrt\nu$ suffices (see Remark \ref{R:diss scale asymptotic}, Remark \ref{R:diss scale asymptotic vort} and Remark \ref{R:Q bmo refined scales}).

\subsection{Quantitative rates and dissipation life-span (Section~\ref{S:rates})}
The analysis developed in the current paper allows to obtain quantitative rates in the Delort class, i.e. when the initial vorticity has singular part of distinguished sign. In this case, it is known that any weak limit $u^\nu  \overset{*}{\rightharpoonup} u$ is a  weak solution to the incompressible Euler equations \cites{delort1991existence,majda1993remarks,scho95,evans1994hardy,vecchi19931}. For convenience we set 
\begin{equation}
    \label{def superlinear beta}
    \mathcal K :=\left\{\beta:\R_+\rightarrow \R_+ \,:\,  \beta\in C^\infty, \, \beta'\geq 0, \, \beta''\geq 0,  \, \lim_{s\rightarrow \infty} \frac{\beta(s)}{s}=\infty \right\}.
\end{equation}

\begin{theorem}\label{T: rates intro}
 Let  $\{u_0^\nu\}_{\nu}\subset L^2(\T^2)$ be a  sequence of divergence-free vector fields with $\{\omega_0^\nu\}_{\nu}\subset \mathcal M (\T^2)$ such that 
 $$
 \sup_{\nu>0 }\left( \|u^\nu_0\|_{L^2_x} + \|\omega^\nu_0\|_{\mathcal M_x}\right)=:M_1 <\infty.
 $$
Assume that $\omega_0^\nu=f^\nu_0 +\mu_0^\nu$ with  $\mu^\nu_0\geq 0$ and $\{f^\nu_0\}_{\nu}\subset L^1(\T^2)$ such that 
\begin{equation}\label{weak comp beta}
    \sup_{\nu>0} \int_{\T^2} \beta \left(|f_0^\nu (x)|\right)\,dx=:M_2<\infty \qquad \text{for some } \beta\in \mathcal K.
\end{equation}
Let $\{u^\nu\}_{\nu}$ be the corresponding sequence of Leray--Hopf solutions to \eqref{NS}. Let $G_\beta$ be the function given by Definition \ref{D:inverse_beta}. There exists a constant $C>0$ depending only on $M_1$ and $M_2$, and a value $\nu_0>0$ depending only on $\beta$ such that, for any $\delta\in (0,1)$, it holds 
    \begin{equation}
        \label{diss_rate}
         \nu\int_\delta^{T} \|\nabla u^\nu(t)\|^2_{L^2_x}\,dt\leq C \sqrt{\frac{T}{\delta} \left( G_\beta(\sqrt{\nu}) + \frac{1}{\sqrt{\log \frac{1}{\nu}}}\right)}\qquad \forall 0<\nu<\nu_0,
    \end{equation}
   as soon as
    \begin{equation}\label{restriction on T and nu}
       T \left( G_\beta(\sqrt{\nu}) + \frac{1}{\sqrt{\log \frac{1}{\nu}}}\right)\leq \frac12.
    \end{equation}
    In particular, if in addition $\{u_0^\nu\}_{\nu}\subset L^2(\T^2)$ is strongly compact, we have
\begin{equation}
        \label{diss_vanish_long_times}
        \lim_{\nu\rightarrow 0} \nu\int_0^{T_\nu} \|\nabla u^\nu(t)\|^2_{L^2_x}\,dt=0
    \end{equation}
    for any sequence of positive real numbers $\{T_\nu\}_{\nu}$ such that 
     \begin{equation}
        \label{times assump}
        \lim_{\nu\rightarrow 0} T_\nu \left( G_\beta(\sqrt{\nu}) + \frac{1}{\sqrt{\log \frac{1}{\nu}}}\right)=0.
    \end{equation}
\end{theorem}
In view of the De la Vall\'ee Poussin criterion\footnote{Restricting to smooth functions in \eqref{def superlinear beta} has been done for convenience. Although the criterion is usually stated without the smoothness requirement, the equivalence of the two can be checked by standard approximation arguments.} \cite{K08}*{Theorem 6.19}, the  assumption \eqref{weak comp beta} is equivalent to the weak compactness of $\{f^\nu_0\}_{\nu}\subset L^1 (\T^2)$. Whenever $\{f^\nu_0\}_{\nu}\subset L^p(\T^2)$ is bounded for some $p>1$,  the convergence  $G_\beta(\sqrt{\nu})\rightarrow 0$ is algebraic (see 
Remark \ref{R:algebraic G}). Thus, for sufficiently small $\nu$, it can be absorbed in the logarithmic term.  It is worth noticing that \eqref{diss_rate} gives a quantitative vanishing rate for the dissipation in $[\delta,T]$ depending only on the initial data. The possibility of getting explicit rates for positive times was first pointed out in \cite{ELL_tocome} where the same asymptotic has been obtained. Although the bound \eqref{diss_rate} degenerates as $\delta\rightarrow 0$, when $\{u^\nu_0\}_\nu\subset L^2(\T^2)$  is strongly compact, it can be extended all the way to $\delta=0$ depending on the $L^2(\T^2)$ modulus of continuity of the sequence of initial data. This will be done in Proposition \ref{P: diss rates up to zero},  thus extending  the results from \cite{ELL_tocome}.

The thesis \eqref{diss_vanish_long_times} provides a lower bound of enhanced dissipation. It was already known by \cite{CLLS16} that an $L^p (\T^2)$ bounded sequence of initial vorticities $\{\omega_0^\nu\}_{\nu}$ implies that $T_\nu$ can be chosen such that $\lim_{\nu\rightarrow 0}\nu T_\nu=0$, independently on $p>1$. It is well known (see Remark \ref{R:time scale 1 over nu}) that any time scale $T_\nu\gtrsim \nu^{-1}$ always results into a dissipation of order $1$, no matter the assumption on the initial data\footnote{Besides the trivial case in which the initial data are converging to zero.}. However, for measure initial vorticities, even in the best scenario in which the absolutely continuous part stays bounded in $L^p (\T^2)$ for some $p>1$, there might be the possibility of observing dissipation already at a logarithmic scale of times, thus much faster than $\nu^{-1}$. 
\subsection{Steady fluids (Section~\ref{S:stationary})}
The main motivation for considering the stationary case comes from the wild oscillations in time that might ruin the spatial atomic concentration of the dissipation from part $(b)$ in Corollary \ref{C: pos vort and atomic}. Ruling out the time dependence allows to prove a Lions-type concentration compactness result on $D$ in full generality, which, as already discussed, might have not been expected in view of Proposition \ref{P: no weak lions}. Of course, here it is necessary to introduce an external forcing

\begin{equation}\label{SNS} \tag{SNS}
\left\{\begin{array}{ll}
\div (u^\nu \otimes u^\nu) +\nabla p^\nu=\nu \Delta u^\nu + f^\nu \\
\div u^\nu=0.
\end{array}\right.
\end{equation}

Differently from the non-stationary case considered before, here the external force plays a role similar to that of the velocity. In this case the uniform $L^2(\T^2)$ bounds must be assumed apriori\footnote{The sequence $u^\nu(x_1,x_2):=\nu^{-1}\sin\left(x_2\right) e_1$ solves \eqref{SNS} with $p^\nu=0$ and $f(x_1,x_2)=\sin\left(x_2\right) e_1$. The force is smooth and independent of viscosity, while $\|u^\nu\|_{L^2}=\nu^{-1}$.}.

\begin{theorem}
    \label{T: measure vort steady}
    Let $\{u^\nu\}_\nu, \{f^\nu\}_\nu\subset C^\infty (\T^2)$ be related by \eqref{SNS}. Assume that $u^\nu\rightharpoonup u$ and $f^\nu\rightharpoonup f$ in $L^2(\T^2)$. Consequently, assume that $|u^\nu-u|^2\overset{*}{\rightharpoonup} \Lambda$, $|f^\nu-f|^2\overset{*}{\rightharpoonup} F$ and $\nu |\nabla u^\nu|^2\overset{*}{\rightharpoonup} D$ in $\mathcal M(\T^2)$. Then $ D\ll \Lambda$ and
    \begin{equation}\label{F zero D zero}
   F=0 \qquad \Longrightarrow \qquad D=0.
        \end{equation}
    In addition, assume that $|\omega^\nu|\overset{*}{\rightharpoonup} \Omega$ in $\mathcal M(\T^2)$, and denote by $\mathscr L$ and $\mathscr O$ the sets of atoms of $\Lambda$ and $\Omega$ respectively.  Then $D=D\llcorner \left( \mathscr L \cap \mathscr O\right)$.
\end{theorem}
In particular, the strong compactness of $\{f^\nu\}_\nu$, i.e. $F=0$, implies no dissipation, i.e. $D=0$. It is not clear to the authors if the global result \eqref{F zero D zero} can be upgraded to the local one $D\ll F$ (see Remark \ref{R:D ll F fails}). In this case we are also able to show the sharpness of Theorem \ref{T: measure vort steady} in several aspects, for instance by providing an explicit example in which all the measures $\Lambda, F, \Omega$ and $D$ have an atom at the origin (see Remark \ref{R: sharp stationary}). This is perhaps not surprising because of the freedom in choosing the external force and imposing the loss of $L^2(\T^2)$ compactness on the velocity, which is somehow inconsistent with the time-dependent case where strong compactness is assumed at the initial time. Of course, the dynamical case is much harder and it is very unclear whether sharpness can be proved. The dynamics in the inviscid limit remains poorly understood in its full generality and, among several other things, it is not known whether the compactness of the initial data can provide effective help.

\subsection{Brief review of related literature} The phenomenon of \quotes{anomalous dissipation} has been an active area of research in mathematical fluid dynamics for both the Navier--Stokes equations \cites{jeong2021vortex,brue2023anomalous,colombo2023anomalous,brue2022onsager} and passive scalar \cites{drivas2022anomalous,johansson2024anomalous,colombo2023anomalous,armstrong2025anomalous,burczak2023anomalous}. We briefly review the main recent developments related to our results.

At first glance, one might expect anomalous dissipation in the Navier--Stokes equations to be closely tied to energy conservation in the Euler equations. Indeed, if the weak* limit \( u \) of a vanishing viscosity sequence \( \{u^\nu\}_\nu \) of Navier--Stokes solutions conserves energy, then \( u \) must solve the Euler equations, and anomalous dissipation is precluded. Conversely, if anomalous dissipation happens, the weak* limit must dissipate energy. From this perspective, recent constructions of wild solutions to the two-dimensional Euler equations may be seen as evidence for anomalous dissipation in two dimensions. In~\cite{GR23}, the authors constructed Euler solutions \( u \in C^\gamma(\T^2\times [0,T]) \) for \( \gamma < \frac{1}{3} \) that do not conserve energy, thereby establishing the two-dimensional counterpart of the Onsager's conjecture. Examples of non-conservative solutions with some vorticity regularity have also been given. These include vortex sheets~\cite{mengual2023dissipative}, Hardy spaces~\cite{buck2024non}, Lorentz spaces~\cite{brue2023nonuniqueness}, and very recently \(C^0([0,T]; L^p(\T^2)) \) for some \( p > 1 \)~\cite{brue2024flexibility}. Note that such $p$ must be strictly smaller than $\frac{3}{2}$ since any Euler solution with vorticity in \( L^3([0,T];L^{\frac32}(\T^2) )\) is known to conserve energy~\cites{CLLS16,CCFS08}. Beyond the issue of energy conservation, recent results have demonstrated other forms of pathological behavior for weak solutions to the two-dimensional Euler equations, including non-uniqueness in the presence of forcing~\cites{V1,V2,castro2024proof,dolce2024self,de2024instability}.

However, such wild Euler solutions do not necessarily arise as vanishing viscosity limits of Navier--Stokes flows, and thus do not directly imply anomalous dissipation. Indeed, whenever the initial vorticity lies in \( L^p(\T^2) \) for some \( p > 1 \), the weak* limit of Navier--Stokes solutions always solves the Euler equations and conserves energy~\cites{CLLS16,LMP21}, thereby ruling out anomalous dissipation in this setting. See also \cites{DRP24,ELL_tocome} for the case of vortex sheet initial data with distinguished sign. Moreover, two-dimensional vanishing viscosity limits often exhibit more regularity than the aforementioned wild solutions, displaying features such as regular Lagrangian flows and renormalization properties~\cites{ciampa2021strong,crippa2014renormalized,crippa2017eulerian}.

\section{Tools}
\subsection{Mollification estimates}
Let $B_1\subset \R^2$ be the disk of radius $1$ centered at the origin. We fix a non-negative radial kernel $\rho\in C^\infty_c(B_1)$ such that $\int\rho=1$. Then, for any $\alpha>0$, we define the sequence of mollifiers as 
$$
\rho_\alpha (x):=\frac{1}{\alpha^2}\rho\left( \frac{x}{\alpha}\right).
$$
Let $p\in [1,\infty]$. For any function $f\in L^p(\T^2)$ we set $f_\alpha:=f*\rho_\alpha$. Clearly $f_\alpha\in C^\infty(\T^2)$ and $f_\alpha\rightarrow f$ in $L^p(\T^2)$, if $p<\infty$. Moreover, we have the following standard estimates
\begin{align}
    \|f_\alpha\|_{L^p}&\leq \|f\|_{L^p} \label{moll est 1},\\
       \|\nabla f_\alpha\|_{L^p}&\leq C \alpha^{-1}\|f\|_{L^p}, \label{moll est 2}\\
       \| f_\alpha-f\|_{L^p}&\leq C \alpha \|\nabla f\|_{L^p} \label{moll est 3},
\end{align}
for some constant $C>0$ and all $\alpha>0$.

\subsection{Quantitative equi-integrability}
We start with the following.
\begin{definition}\label{D:inverse_beta}
Let $\mathcal K$ be as in \eqref{def superlinear beta}. For any $\beta\in \mathcal K$ we set $g_\beta$ to be the inverse of the map $s\mapsto \frac{s}{\beta(s)}$ and, consequently, $G_\beta$ to be the inverse of the map $s\mapsto \frac{s}{g_\beta(s)}$.
\end{definition}
Although the maps $s\mapsto \frac{s}{\beta(s)}$ or $s\mapsto \frac{s}{g_\beta(s)}$  might not be invertible for all $s\in \mathbb{R}_+$, in the later analysis, we will only require their invertibility for a certain range of $s$.  The next simple proposition makes the above definition sensible for such ranges. 

\begin{proposition}
    \label{P:inverse beta well defined}
    Let $\beta\in \mathcal{K}$. There exist $c_1,c_2,c_3>0$ such that the following hold. The function $g_\beta:[0,c_1]\rightarrow [c_2,\infty)$ is well-defined, continuous, surjective and strictly decreasing. The function $G_\beta:[0,c_3]\rightarrow [0,c_1]$ is well-defined,  continuous, surjective and strictly increasing. Moreover $G_\beta(0)=0$.
\end{proposition}
\begin{proof}
    Consider the map $s\mapsto \frac{\beta(s)}{s}$. Its derivative is given by
    $$
    \frac{s\beta'(s)-\beta(s)}{s^2},
    $$
    which is positive if and only if $f(s):=s\beta'(s)-\beta(s)>0$. Since $\beta$ is super-linear at infinity, there must be $s_0>0$ such that $f(s_0)>0$. Moreover, $f'(s)=s\beta''(s)\geq 0$. Thus $f(s)>0$ for all $s\geq s_0$. In particular, the map $s\mapsto \frac{s}{\beta(s)}$ is strictly decreasing, mapping $[s_0,\infty)$ onto $\left[0,\frac{s_0}{\beta(s_0)}\right]$.  Setting $c_1:=\frac{s_0}{\beta(s_0)}$ and $c_2=s_0$ we get that $g_\beta:[0,c_1]\rightarrow [c_2,\infty)$ is well-defined, continuous, surjective and strictly decreasing. It follows that $ \frac{s}{g_\beta (s)}$ is continuous, strictly increasing, mapping $[0,c_1]$ onto $\left[0,\frac{c_1}{g_\beta(c_1)}\right]$ and it vanishes at $s=0$. Then, setting $c_3=\frac{c_1}{g_\beta(c_1)}$, we conclude that the function $G_\beta:[0,c_3]\rightarrow [0,c_1]$ is well-defined,  continuous, surjective, strictly increasing and $G_\beta(0)=0$.
\end{proof}
The function $G_\beta$ quantifies the decay on small balls.
 \begin{lemma}\label{L:beta_int_on_ball}
 Let $\{f_n\}_n\subset L^\infty([0,T];L^1(\T^2))$ be such that 
 $$
 \sup_{t,n} \int_{\T^2} \beta(|f_n(x,t)|)\,dx=:M<\infty\qquad \text{for some } \beta\in \mathcal K.
 $$
Let $G_\beta$ be the function from Definition \ref{D:inverse_beta}. There exist $C>0$ depending only on $M$ and $r_0>0$ depending only on $\beta$ such that
 \begin{equation}\label{fn decay on ball}
     \sup_{x,t,n} \int_{B_r(x)} |f_n(y,t)|\,dy \leq CG_\beta (r^2)\qquad \forall 0<r<r_0.
 \end{equation}
 \end{lemma}
\begin{proof}
    Let $\eps>0$ be arbitrary, but smaller than the value $c_1$ from Proposition \ref{P:inverse beta well defined}. In the notation of Definition \ref{D:inverse_beta} we have 
    \begin{equation}\label{inverting_beta}
    \frac{s}{\beta(s)}\leq \eps \qquad  \forall s\geq  g_\beta(\eps).
    \end{equation}
    Let $x\in \T^2$. We split
    \begin{align}
        \int_{B_r(x)} |f_n(y,t)|\,dy&=\int_{B_r(x)\cap \{|f_n|<g_\beta(\eps)\}} |f_n(y,t)|\,dy+\int_{B_r(x)\cap \{|f_n|\geq g_\beta(\eps)\}} |f_n(y,t)|\,dy\\
        &\leq g_\beta(\eps) \pi r^2+ \int_{B_r(x)\cap \{|f_n|\geq g_\beta(\eps)\}} |f_n(y,t)|\,dy.
    \end{align}
    Moreover, by \eqref{inverting_beta} we get 
    \begin{align}
        \int_{B_r(x)\cap \{|f_n|\geq g_\beta(\eps)\}} |f_n(y,t)|\,dy&= \int_{B_r(x)\cap \{|f_n|\geq g_\beta(\eps)\}}\frac{|f_n(y,t)|}{\beta(|f_n(y,t)|)}\beta(|f_n(y,t)|)\,dy\\
        & \leq \eps \sup_{n,t} \int_{\T^2} \beta(|f_n(y,t)|)\,dy\\
        &\leq M\eps,
    \end{align}
    which yields to 
    $$
    \int_{B_r(x)} |f_n(y,t)|\,dy \leq C \left(g_\beta (\eps) r^2+ \eps \right),
    $$
    for some constant $C>0$ depending only on $M$.
    By optimizing in $\eps$ we find, for $r^2$ smaller than the value $c_3$ from Proposition \ref{P:inverse beta well defined}, $\eps_{\rm opt}= G_\beta (r^2)$. The thesis follows by choosing $r_0=\sqrt{c_3}$.
\end{proof}

\begin{remark}
    \label{R:algebraic G}
    If $\{f_n\}_n\subset L^p(\T^2)$ is bounded for some $p>1$, we can take $\beta(s)=s^p$. In this case $g_\beta(s)=s^{\frac{1}{1-p}}$ and $G_\beta(s)=s^{\frac{p-1}{p}}$. Thus \eqref{fn decay on ball} is coherent with what could have been obtained by the H\"older inequality. In particular, it is sharp.
\end{remark}

\subsection{Curves of measures and absolute continuity}
Let $I\subset \R$ be an interval. We recall that a map $t\mapsto \mu_t$ from $I$ to $\mathcal M(\T^2)$ is said to be weakly measurable if $t\mapsto \langle \mu_t,\varphi\rangle$ is measurable for any $\varphi \in  C^0(\T^2)$. Then, we say that $\mu\in L^p(I;\mathcal M(\T^2))$ if $\mu=\mu_t\otimes dt$ for a weakly measurable map $t\mapsto \mu_t$ such that $\mu_t(\T^2)\in L^p(I)$. Clearly, any  $\mu\in L^p(I;\mathcal M(\T^2))$ can be identified with an element of $\mathcal{M} (\T^2\times I)$.

\begin{lemma}
    \label{L:slice_measures}
    Let $\mu,\lambda\in L^1(I;\mathcal{M}(\T^2))$ be such that $\mu\ll \lambda$. Then $\mu_t\ll\lambda_t$ for a.e. $t\in I$.
\end{lemma}
\begin{proof}
    By the Radon--Nikodym theorem we find a Borel function $g\in L^1(\T^2\times I;\lambda)$ such that $d\mu=g \,d\lambda$. Thus, for any choice of $\psi\in C^0(\T^2)$ and $\eta\in C^0(I)$, we deduce
    $$
    \int_I \eta(t)\left(\int_{\T^2} \psi(x) \,d\mu_t(x)\right)dt = \int_I \eta(t)\left(  \int_{\T^2} \psi(x) g(x,t) \,d\lambda_t(x) \right)dt.
    $$
    Since $C^0(\T^2)$ is separable, by a standard argument we find a negligible set of times  $\mathcal{N}\subset I$ such that, for all $t\in \mathcal N^c$, it holds
        $$
\int_{\T^2} \psi(x) \,d\mu_t(x) =   \int_{\T^2} \psi(x) g(x,t) \,d\lambda_t(x)  \qquad \forall \psi\in C^0(\T^2).
    $$
By the arbitrariness of $\psi$ we obtain
           $$
\int_{A}  \,d\mu_t(x) =   \int_{A} g(x,t) \,d\lambda_t(x)  \qquad \forall A\subset \T^2 \text{ Borel, } \forall t\in \mathcal{N}^c,
    $$
    which yields to $\mu_t \ll \lambda_t$ for all $t\in \mathcal N^c$.
\end{proof}

\subsection{Some remarks on Navier--Stokes}
Given a vector field $u:\T^2\times [0,T]\rightarrow \R^2$ we denote by $E_u$ its kinetic energy, i.e. 
$$
E_u(t):=\frac{1}{2}\int_{\T^2}|u(x,t)|^2\, dx.
$$
As already said, a direct consequence of \eqref{NS en bal} is that a sequence of solutions to \eqref{NS} emanating from an $L^2(\T^2)$ bounded sequence of initial data $\{u^\nu_0\}_\nu$, stays bounded in $L^\infty([0,T];L^2(\T^2))$. We can thus assume $u^\nu\overset{*}{\rightharpoonup}u$ in $L^\infty([0,T];L^2(\T^2))$.  If  $u^\nu_0\rightarrow u_0$ in $L^2(\T^2)$, this yields to
\begin{equation}\label{energy admissibility new}
    E_u(t)\leq E_{u_0}\qquad \text{for a.e. } t\in [0,T].
\end{equation}
Denote by $L^2_{\rm w} (\T^2)$ the space of $L^2(\T^2)$ functions endowed with the weak topology. Although not essential for our purposes, we recall some basic properties of the weak limit. 
\begin{lemma}
\label{L:energy right cont}
Let $u_0\in L^2(\T^2)$ be given. Assume $u\in C^0([0,T];L^2_{\rm w} (\T^2))$ satisfies $u(t_n)\rightharpoonup u_0$ in $L^2(\T^2)$ as $t_n\rightarrow 0$ and $E_u(t)\leq E_{u_0}$ for all $t\in [0,T]$. Then its kinetic energy $E_u$ is continuous from the right at $t=0$.
\end{lemma}

\begin{remark}\label{R:diss_sol}
Assume $u^\nu_0\rightarrow u_0$ in $L^2(\T^2)$. Any limit $u^{\nu} \overset{*}{\rightharpoonup} u$ in $L^\infty([0,T]; L^2(\mathbb{T}^2))$ of a sequence of Leray--Hopf solutions to \eqref{NS} can be redefined on a negligible set of times so that $u\in C^0([0,T];L^2_{\rm w} (\T^2))$. Indeed, any such limit is a \quotes{dissipative} solution in the sense of Lions \cite{Lions_book}*{Chapter 4}. It follows that \eqref{energy admissibility new} can be upgraded to hold for all $t\in [0,T]$. In particular, Lemma \ref{L:energy right cont} implies that $E_u$ is right-continuous at $t=0$. For the representative $u\in C^0([0,T];L^2_{\rm w} (\T^2))$, we can further assume that $u^\nu(t)\rightharpoonup u(t)$ in $L^2(\T^2)$ for all $t\in [0,T]$. Indeed, by the uniform bound of $\{u^{\nu}\}_\nu$ in $L^\infty([0,T]; L^2(\mathbb{T}^2))$, the Navier--Stokes equations automatically imply that $\{u^{\nu}\}_\nu$ stays bounded in $\Lip ([0,T]; H^{-N}(\mathbb{T}^2))$ for a sufficiently large $N\in \N$, from which the claim follows by the Aubin--Lions lemma and the density of $C^\infty(\T^2)$ in $L^2(\T^2)$.
\end{remark}

We will make use of the following classical estimates for two-dimensional viscous fluids 
\begin{align}
\|\omega^\nu(t)\|_{L^1_x}&\leq\|\omega^\nu_0\|_{\mathcal M_x}\qquad \forall t>0,\label{vortic stays a measure}\\
\|\omega^\nu(t)\|^2_{L^2_x}&\leq \frac{\|u^\nu_0\|^2_{L^2_x}}{2t \nu}\qquad \forall t>0. \label{D bounded positive times}
\end{align}
For the proof see for instance \cite{DRP24}*{Proposition 2.4} and \cite{DRP24}*{Lemma 3.1}.
 Moreover, for measure initial vorticity, it is possible to decompose the dynamics of the absolutely continuous and singular parts. When the singular part has distinguished sign, the decomposition also comes with nice bounds.

\begin{proposition}\label{P:split_vort_L1_pos}
Let $\{u^\nu_0\}_{\nu}\subset L^2(\T^2)$ be a sequence of divergence-free vector fields such that $\{\omega^\nu_0\}_{\nu}\subset \mathcal M(\T^2)$ admits a decomposition $\omega_0^\nu =f^\nu_0 + \mu_0^\nu$ with  $\left\{ f^\nu_0\right\}_{\nu}\subset L^1(\T^2)$ and  $\mu^\nu_0 \geq 0$. Let $\{u^\nu\}_{\nu}$ be the corresponding sequence of Leray--Hopf solutions and denote by $\{\omega^\nu\}_{\nu}$ the corresponding sequence of  vorticities. There exists a decomposition $
\omega^\nu=f^\nu + \mu^\nu$
such that
\begin{equation}
\{f^{\nu}\}_{\nu}\subset L^\infty([0,T];L^1(\T^2)) \qquad \text{and} \qquad \{\mu^\nu\}_{\nu}\subset L^\infty([0,T];\mathcal M(\T^2)), \, \mu^\nu\geq 0.
\end{equation}
In particular, it holds 
\begin{equation}\label{ineq_delort_pos}
|\omega^\nu|\leq 2|f^\nu|+\omega^\nu.
\end{equation}
In addition, if $\mathcal K$ is the set defined in \eqref{def superlinear beta}, for every $\beta\in \mathcal K$ it holds 
\begin{equation}\label{beta_f_estimate}
\int_{\T^2}\beta(|f^\nu(x,t)|)\,dx \leq \int_{\T^2} \beta(|f^\nu_0(x)|)\,dx\qquad \forall t>0.
\end{equation}
\end{proposition}
The above proposition follows by the proof of \cite{DRP24}*{Proposition 3.2}. We conclude this section with the following.
\begin{proposition}
    \label{P: D and tilde D are equiv}
    Let $\{u^\nu_0\}_{\nu}\subset L^2(\T^2)$ be a bounded sequence of divergence-free vector fields and let $\{u^\nu\}_\nu$ be the corresponding sequence of Leray--Hopf solutions to \eqref{NS}. Assume that $\nu|\nabla u^{\nu}|^2\overset{*}{\rightharpoonup} D$ and $\nu|\omega^{\nu}|^2\overset{*}{\rightharpoonup} \tilde D$ in $\mathcal{M}(\mathbb{T}^2\times [0,T])$. There exists a constant $C>0$ such that 
    \begin{equation}
        \label{quantitative equivalence D and tilde D}
        \frac{1}{C} \tilde D (A)\leq D(A)\leq C \tilde D(A) \qquad \forall A\subset\T^2\times [0,T], \, A \text{ Borel}.
    \end{equation}
\end{proposition}
\begin{proof}
Note that  $|\omega^\nu|^2\leq 2|\nabla u^\nu|^2$ holds point-wise in space-time. Then the lower bound in \eqref{quantitative equivalence D and tilde D} directly follows. We are left to prove $D(A)\leq C \tilde D(A)$. The goal is to localize the Calder\'on--Zygmund estimate relating $\nabla u^{\nu}$ and $\omega^{\nu}$. To this end, let $\varphi\in C^1(\T^2\times [0,T])$ be arbitrary. By setting $\tilde u^{\nu}:=u^{\nu}\varphi$,  we have
\begin{equation}
    \label{general_CZ}
    \nu\int_0^T\int_{\T^2} |\nabla \tilde u^{\nu}|^2\leq C \nu\left(\int_0^T\int_{\T^2} |\curl \tilde u^{\nu}|^2+\int_0^T\int_{\T^2} |\div \tilde u^{\nu}|^2 \right).
\end{equation}
Note that $\nabla \tilde u^{\nu}=\varphi \nabla u^{\nu} + u^{\nu}\otimes \nabla \varphi$. Thus we can expand the left-hand-side in \eqref{general_CZ} as 
$$
\underbrace{\nu\int_0^T \int_{\T^2} |\varphi|^2 |\nabla u^{\nu}|^2}_{=:I_\nu}+\underbrace{\nu\int_0^T \int_{\T^2} |u^{\nu}\otimes \nabla \varphi|^2}_{=:II_\nu}+\underbrace{2\nu\int_0^T \int_{\T^2} \varphi\nabla u^{\nu} :u^{\nu}\otimes \nabla\varphi}_{=:III_\nu}.
$$

Clearly $I_\nu\rightarrow \langle D, |\varphi|^2\rangle$ by assumption. Moreover, by \eqref{NS en bal} we get
\begin{align}
 II_\nu\leq C \nu \|u^{\nu}_0\|_{L^2_x}^2\rightarrow 0\qquad \text{and}\qquad 
\left| III_\nu \right|\leq C \sqrt{\nu}  \|u^{\nu}_0\|_{L^2_x}^2\rightarrow 0.
\end{align}
These prove
$$
\nu\int_0^T\int_{\T^2} |\nabla \tilde u^{\nu}|^2\rightarrow \langle D, |\varphi|^2\rangle.
$$
Similarly, by expanding $\curl \tilde u^{\nu} =\omega^{\nu}\varphi + u^{\nu}\cdot \nabla^\perp \varphi$ and $\div \tilde u^{\nu}=u^{\nu}\cdot \nabla \varphi$ we obtain
$$
\nu\left(\int_0^T\int_{\T^2} |\curl \tilde u^{\nu}|^2+\int_0^T\int_{\T^2} |\div \tilde u^{\nu}|^2 \right)\rightarrow \langle \tilde D, |\varphi|^2\rangle.
$$
We have thus proved that $\langle  D, |\varphi|^2\rangle\leq C\langle \tilde D, |\varphi|^2\rangle$ for all $\varphi\in C^1(\T^2\times [0,T])$, from which the upper bound in \eqref{quantitative equivalence D and tilde D} follows.
\end{proof}

\section{Dissipation, defect and vorticity measures}\label{Measure_comparison}
Here we discuss the proofs of Theorem \ref{T: general leray}, Theorem \ref{T: measure vort} and Corollary \ref{C: pos vort and atomic}. As most of the results proved in this paper, they all build on the following two propositions: the first (Proposition \ref{P:dissipation short times quantitative}) dealing with the dissipation for short times, while the second (Proposition \ref{P:dissipation positive times}) for strictly positive times. The two regimes are quite different.

The next proposition gives a quantitative equi-continuity of the dissipation for short times in terms of the $L^2(\T^2)$ modulus of continuity of the initial data. As it will be clear from the proof, stronger assumptions, e.g. a uniform bound of the initial data in $C^\sigma(\T^2)$, would lead to stronger conclusions.

\begin{proposition}
    \label{P:dissipation short times quantitative}
      Let $\{u^\nu_0\}_\nu\subset L^2(\T^2)$ be a bounded sequence of divergence-free vector fields. Denote by $u^\nu_{0,\eps}:=u^\nu_0*\rho_\eps$ the mollification of $u^\nu_0$ and define
      \begin{equation}
          \label{modulus of comapctness}
          \Phi(\eps):=\sup_{\nu>0}\|u^\nu_{0,\eps}-u^\nu_0\|_{L^2_x}.
      \end{equation}
      Then, denoting by $\{u^\nu\}_\nu$ the corresponding sequence of Leray--Hopf solutions to \eqref{NS}, there exists a constant $C>0$ such that 
      \begin{equation}
      \|u^\nu(\delta)-u^\nu_0\|^2_{L^2_x}\leq C \left( \Phi(\eps)+\frac{\delta}{\eps^2}\right) \qquad \forall \eps,\delta,\nu\in (0,1). \label{compactness short times quant}
       \end{equation}
      Consequently
      $$
      \nu\int_0^\delta \|\nabla u^\nu(t)\|^2_{L^2_x}\,dt\leq C \sqrt{\Phi(\eps)+\frac{\delta}{\eps^2}} \qquad \forall \eps,\delta,\nu\in (0,1).
      $$
      In particular, if $\{u^\nu_0\}_\nu\subset L^2(\T^2)$ is strongly compact\footnote{Note that $\Phi(\eps)\rightarrow 0$ if and only if $\{u^\nu_0\}_\nu\subset L^2(\T^2)$ is strongly compact.}, for any $\eps>0$ there exists $\delta>0$ such that 
      \begin{equation}
          \label{diss short time equicont}
          \sup_{\nu> 0} \nu\int_0^\delta \|\nabla u^\nu(t)\|^2_{L^2_x}\,dt<\eps.
      \end{equation}
\end{proposition}

\begin{proof}
    Since $\|u^\nu(\delta)\|_{L^2_x}\leq \|u^\nu_0\|_{L^2_x}$, we bound 
    \begin{align}
        \|u^\nu(\delta)-u^\nu_0\|^2_{L^2_x}&= \|u^\nu(\delta)\|^2_{L^2_x} -  \|u^\nu_0\|^2_{L^2_x} +2\int_{\T^2} u^\nu_0(x)\cdot ( u^\nu_0(x)-u^\nu(x,\delta))\,dx\\
        &\leq 2\int_{\T^2} u^\nu_0(x)\cdot ( u^\nu_0(x)-u^\nu(x,\delta))\,dx\\
        &\leq C\Phi(\eps) + 2\underbrace{\int_{\T^2} u^\nu_{0,\eps}(x)\cdot ( u^\nu_0(x)-u^\nu(x,\delta))\,dx}_{I}.\label{jaemin smart split}
    \end{align}
    By using $u^\nu_{0,\eps}$ as a test function for \eqref{NS} we get
    \begin{align}
        |I|&=\left| \int_0^\delta \int_{\T^2}  u^\nu\otimes u^\nu :\nabla u^\nu_{0,\eps}+\nu \int_0^\delta \int_{\T^2} u^\nu\cdot \Delta u^\nu_{0,\eps}\right|\\
&\leq \int_0^\delta \|u^\nu(t)\|^2_{L^2_x}\|\nabla u^\nu_{0,\eps}\|_{L^\infty_x}\,dt +\nu \int_0^\delta \|u^\nu(t)\|_{L^2_x}\|\Delta u^\nu_{0,\eps}\|_{L^2_x}\,dt\\
&\leq C\delta \left( \|\nabla u^\nu_{0,\eps}\|_{L^\infty_x}+\|\Delta u^\nu_{0,\eps}\|_{L^2_x}\right).
    \end{align}
    These last two terms can be bounded respectively as 
    $$
    \|\nabla u^\nu_{0,\eps}\|_{L^\infty_x}=\|u^\nu_0*\nabla \rho_\eps\|_{L^\infty_x}\leq \|u^\nu_0\|_{L^2_x}\|\nabla \rho_\eps\|_{L^2_x}\leq C \eps^{-2}
    $$
    and 
    $$
    \|\Delta u^\nu_{0,\eps}\|_{L^2_x}=\|u^\nu_0*\Delta \rho_\eps\|_{L^2_x}\leq \|u^\nu_0\|_{L^2_x}\|\Delta \rho_\eps\|_{L^1_x}\leq C \eps^{-2}.
    $$
    Therefore, we deduce $|I|\leq C \delta \eps^{-2}$. Plugging this back into \eqref{jaemin smart split} yields to \eqref{compactness short times quant}. By the energy balance \eqref{NS en bal} we then obtain 
    \begin{align}
    \nu\int_0^\delta \|\nabla u^\nu(t)\|^2_{L^2_x}\,dt&=\|u^\nu_0\|^2_{L^2_x}-\|u^\nu(\delta)\|^2_{L^2_x}\\
    &\leq \|u^\nu_0-u^\nu(\delta)\|_{L^2_x}\|u^\nu_0+u^\nu(\delta)\|_{L^2_x}\\
    &\leq C \sqrt{\Phi(\eps)+\frac{\delta}{\eps^2}}.
       \end{align}
       When $\{u^\nu_0\}_\nu\subset L^2(\T^2)$ is strongly compact we have $\Phi(\eps)\rightarrow 0$ as $\eps\rightarrow 0$, from which we conclude the validity of \eqref{diss short time equicont}.
\end{proof}
A consequence of Proposition \ref{P:dissipation short times quantitative} is that the strong $L^2(\T^2)$ compactness of the initial data prohibits the dissipation to instantaneously happen. Without the initial compactness one can construct a counter-example (see Remark \ref{R:D initial time concentration}) on the whole space by a simple scaling analysis. However, there are also examples where the initial compactness does not hold but the equi-continuity in time \eqref{diss short time equicont} is still true (see Remark \ref{R:sharp atomic dynamics}).

\begin{proposition}
    \label{P:dissipation positive times}
         Let  $\{u_0^\nu\}_{\nu}\subset L^2(\T^2)$ be a sequence of divergence-free vector fields and let $\{u^\nu\}_{\nu}$ be the corresponding sequence of Leray--Hopf solutions to \eqref{NS}. Then, for any $\delta,\nu>0$
         \begin{equation}
             \label{higher order dynamical}
              \nu^2\int_\delta^T  \|\nabla \omega^\nu (t)\|^2_{L^2_x}\,dt \leq \frac{\|u^\nu_0\|^2_{L^2_x}}{\delta}.
         \end{equation}
         In particular, denoting by $\omega^\nu_\alpha:=\omega^\nu*\rho_\alpha$ the space mollification of $\omega^\nu$, there exists a constant $C>0$ such that 
         \begin{equation}
             \label{fundamental split dissipation}
        \nu\int_\delta^T\|\omega^\nu(t)\|^2_{L^2_x}\,dt\leq \nu \int_\delta^T\int_{\T^2} \omega^\nu(x,t) \omega^\nu_\alpha(x,t) \, dxdt + C\frac{\alpha}{\sqrt{\nu\delta}} \|u^\nu_0\|^2_{L^2_x}
         \end{equation}
         for all $\delta,\nu,\alpha>0$.
\end{proposition}

\begin{proof}
By taking the $\curl$ of the first equation in \eqref{NS} we get 
$$
\partial_t\omega^\nu +u^\nu \cdot \nabla \omega^\nu =\nu \Delta \omega^\nu.
$$
It follows 
\begin{align}
\frac12 \|\omega^\nu(T)\|^2_{L^2_x}+\nu \int_s^T\|\nabla \omega^\nu(t)\|^2_{L^2}\,dt= \frac12 \|\omega^\nu(s)\|^2_{L^2_x}\qquad \forall0<s<T.
\end{align}
By integrating in $\int_0^T\cdot \,ds$ and using \eqref{NS en bal} we obtain 
\begin{align}
\nu\int_0^T t\|\nabla\omega^\nu(t)\|^2_{L^2_x}\,dt &=\int_0^T\int_s^T \|\nabla\omega^\nu(t)\|^2_{L^2_x}\,dtds\\
&\leq \frac12 \int_0^T\|\omega^\nu(s)\|^2_{L^2_x}\,ds \\
&\leq \frac{1}{4\nu} \|u^\nu_0\|^2_{L^2_x}.
\end{align}
Then \eqref{higher order dynamical} immediately follows. To obtain \eqref{fundamental split dissipation} we simply split 
$$
\nu\int_\delta^T\int_{\T^2}|\omega^\nu|^2=\nu\int_\delta^T\int_{\T^2} \omega^\nu\omega^\nu_\alpha+\nu\int_\delta^T\int_{\T^2} \omega^\nu(\omega^\nu-\omega^\nu_\alpha),
$$
and then estimate the very last term by \eqref{moll est 3}, \eqref{higher order dynamical} and \eqref{NS en bal} as 
\begin{align}
    \nu\int_\delta^T\int_{\T^2} \omega^\nu(\omega^\nu-\omega^\nu_\alpha)&\leq \nu\alpha\int_\delta^T\|\omega^\nu(t)\|_{L^2_x}\|\nabla \omega^\nu(t)\|_{L^2_x}\,dt\leq C \frac{\alpha}{\sqrt{\nu\delta}} \|u^\nu_0\|^2_{L^2_x}.
\end{align}
\end{proof}

\begin{remark}
    \label{R:refined bound for rates}
    If in the above proof, instead of writing $|\omega^\nu|^2=\omega^\nu\omega^\nu_\alpha +\omega^\nu(\omega^\nu-\omega^\nu_\alpha)$, we use $|\omega^\nu|^2\leq 2(|\omega^\nu_\alpha|^2+|\omega^\nu-\omega^\nu_\alpha|^2)$, we can replace \eqref{fundamental split dissipation} with 
    \begin{equation}
             \label{refined fundamental split dissipation}
        \nu\int_\delta^T\|\omega^\nu(t)\|^2_{L^2_x}\,dt\leq 2 \nu \int_\delta^T\int_{\T^2}  |\omega^\nu_\alpha(x,t)|^2 \, dxdt + C\frac{\alpha^2}{\nu\delta} \|u^\nu_0\|^2_{L^2_x}\qquad \forall \delta,\nu,\alpha>0.
         \end{equation}
         This will be used in Section \ref{S:rates} to obtain better rates.
\end{remark}

We are now ready to prove Theorem \ref{T: general leray}, Theorem \ref{T: measure vort} and Corollary \ref{C: pos vort and atomic}.
\begin{proof}[Proof of Theorem \ref{T: general leray}]
  We divide the proof into steps.

  \underline{\textsc{Step 1}}: $D\in L^1([0,T];\mathcal M(\T^2))$.

  By \eqref{D bounded positive times} we deduce that $\{\nu|\nabla u^\nu|^2\}_\nu\subset L^\infty_{\rm loc} ((0,T];L^1(\T^2))$ is bounded. Then $D\in L^\infty_{\rm loc} ((0,T];\mathcal M(\T^2))$ necessarily. This means that 
  \begin{equation}
      \label{proof ok for pos times}
      \int_\delta^T \int_{\T^2} \varphi \, dD= \int_\delta^T \left( \int_{\T^2}\varphi(x,t)\,dD_t(x)\right)dt\qquad \forall \varphi \in C^0(\T^2\times [0,T]),\, \forall\delta>0,
  \end{equation}
  for some weakly measurable map $t\mapsto D_t$, $D_t(\T^2)\in L^\infty_{\rm loc}((0,T])$. The goal is to show that \eqref{proof ok for pos times} holds for $\delta=0$. This is equivalent to say that $D$ does not concentrate at the initial time. 

 By the lower semi-continuity of the weak* convergence of measures on open sets
  \begin{align}
      \int_\delta^T D_t(\T^2)\,dt\leq \liminf_{\nu\rightarrow 0}  \nu\int_\delta^T \int_{\T^2} |
      \nabla u^\nu|^2\leq \frac{\sup_{\nu>0}\|u^\nu_0\|^2_{L^2_x}}{2}<\infty.
  \end{align}
  Hence, letting $\delta\rightarrow 0$, we obtain $D_t(\T^2)\in L^1([0,T])$. For any $\varphi\in C^0(\T^2\times [0,T])$ such that $|\varphi|\leq 1$, we split 
  \begin{align}
      \left| \int_0^T \int_{\T^2} \varphi \, dD - \int_0^T \left( \int_{\T^2}\varphi(x,t)\,dD_t(x)\right)dt\right|&=  \left| \int_0^\delta \int_{\T^2} \varphi \, dD - \int_0^\delta \left( \int_{\T^2}\varphi(x,t)\,dD_t(x)\right)dt\right|\\
      &\leq D(\T^2\times [0,\delta])+\int_0^\delta D_t(\T^2)\,dt,\label{preparing D to be lebesgue}
  \end{align}
  where to obtain the first identity we have used \eqref{proof ok for pos times}. By $D_t(\T^2)\in L^1([0,T])$ we have 
  $$
  \lim_{\delta\rightarrow 0} \int_0^\delta D_t(\T^2)\,dt=0.
  $$
  Moreover
  $$
  D(\T^2\times [0,\delta])\leq \limsup_{\nu\rightarrow 0} \nu\int_0^{2\delta} \|\nabla u^\nu (t)\|^2_{L^2_x}
  \,dt,
  $$
  which vanishes as $\delta\rightarrow 0$ by Proposition \ref{P:dissipation short times quantitative} thanks to the strong compactness of the initial data in $L^2(\T^2)$. Thus, by letting $\delta\rightarrow 0$ in \eqref{preparing D to be lebesgue} we conclude $D=D_t\otimes dt$ as elements in $\mathcal M(\T^2\times [0,T])$.

  Since $D\in L^1([0,T];\mathcal M(\T^2))$, to prove that $D_t\ll \Lambda_t$ for a.e. $t\in [0,T]$, it is enough to prove that, for any $\delta>0$ it holds 
  \begin{equation}
      \label{D ll L for pos times}
      D_t\ll \Lambda_t \qquad \text{for a.e. } t\in [\delta,T].
  \end{equation}
Since from now on $\delta>0$ will be fixed, we will not keep track of it in all the estimates below. Most of them degenerate as $\delta\rightarrow 0$.

\underline{\textsc{Step 2}}: \textsc{Dissipation splitting}. 

 Let $\varphi\in C^\infty(\T^2\times [0,T])$ be an arbitrary non-negative test function. Integrating by parts, we split the dissipation into three terms
    \begin{align}
        \nu\int_\delta^T \int_{\T^2}  |\nabla u^\nu|^2 \varphi &=-\nu \int_\delta^T \int_{\T^2} u^\nu \cdot \Delta u^\nu \varphi - \nu \int_\delta^T \int_{\T^2} \nabla \varphi \cdot \nabla \frac{|u^\nu|^2}{2}\\
        &= -\underbrace{\nu \int_\delta^T \int_{\T^2} (u^\nu-u) \cdot \Delta u^\nu \varphi}_{I_\nu}  -\underbrace{\nu \int_\delta^T \int_{\T^2} u \cdot \Delta u^\nu \varphi}_{II_\nu}\\
        & \quad + \underbrace{\nu \int_\delta^T \int_{\T^2} \frac{|u^\nu|^2}{2} \Delta \varphi}_{III_\nu},\label{eq_splitting_new_proof_lambda_dynam}
    \end{align}
    where $u$ is the weak* limit of $\{u^\nu\}_\nu$ as in the statement of the theorem.  The term $I_\nu$ is the main contribution which is related to the defect measure $\Lambda$, while the terms $II_\nu$ and $III_\nu$ will be shown to be negligible as $\nu\rightarrow 0$. 

\underline{\textsc{Step 3}}: $II_\nu,III_\nu\rightarrow 0$.

Clearly 
\begin{equation}
        \limsup_{\nu\rightarrow 0} \big| III_\nu \big| \leq C  \limsup_{\nu\rightarrow 0} \nu \int_\delta^T \|u^\nu(t)\|_{L^2_x}^2\,dt\leq C \limsup_{\nu\rightarrow 0} \nu  \|u^\nu_0\|_{L^2_x}^2=0.
    \end{equation}

Similarly 
$$
\nu \int_\delta^T\int_{\T^2}\psi \cdot \Delta u^\nu =\nu \int_\delta^T\int_{\T^2} u^\nu\cdot \Delta \psi\rightarrow 0 \qquad \forall \psi\in C^\infty(\T^2\times [\delta,T]).
$$
Since $\{\nu \Delta u^\nu\}_\nu\subset L^2(\T^2\times [\delta,T])$ is bounded by \eqref{higher order dynamical}, this shows  $\nu \Delta u^\nu\rightharpoonup 0$ in $L^2(\T^2\times [\delta,T])$. Then $II_\nu\rightarrow 0$.

\underline{\textsc{Step 4}}: $I_\nu\sim \Lambda$ \textsc{and conclusion}.

By Cauchy--Schwarz  and \eqref{higher order dynamical} we get 
\begin{align}
    \big| I_\nu\big| &\leq  \nu \left( \int_\delta^T \int_{\T^2} \varphi^2 |u^\nu -u|^2 \right)^{\frac{1}{2}}\left(\int_\delta^T \|\nabla \omega^\nu (t)\|_2^2\,dt \right)^\frac{1}{2}\leq C\left( \int_\delta^T \int_{\T^2} \varphi^2 |u^\nu -u|^2 \right)^{\frac{1}{2}}.
\end{align}
Thus, by letting $\nu\rightarrow 0$ in \eqref{eq_splitting_new_proof_lambda_dynam}, we achieve 
\begin{equation}\label{D and L quant pos times}
\int_\delta^T \int_{\T^2} \varphi \,dD\leq C \left( \int_\delta^T \int_{\T^2} \varphi^2 \,d\Lambda \right)^{\frac{1}{2}} \qquad \forall \varphi \in C^\infty(\T^2\times [0,T]), \,\forall\delta>0.
\end{equation}
This shows that $D\ll \Lambda$ as measures on $\T^2\times [\delta,T]$, from which we conclude the validity of \eqref{D ll L for pos times} by Lemma \ref{L:slice_measures}.
\end{proof}

 \begin{remark}
     \label{R:D vs Lambda quantitative}   
     For any fixed $\delta>0$, the estimate \eqref{D and L quant pos times} gives $D(A)\leq C_\delta \Lambda^{\sfrac{1}{2}}(A)$ for all Borel sets $A\subset \T^2\times [\delta,T]$. Thus, for positive times, the absolute continuity is quantitative. The constant $C_\delta\sim \frac{1}{\sqrt{\delta}}$ degenerates as $\delta\rightarrow 0$.
 \end{remark}
\begin{proof}[Proof of Theorem \ref{T: measure vort}]
    The first two claims $D\in L^1([0,T];\mathcal M(\T^2))$ and $D_t\ll \Lambda_t$ have already been proved in Theorem \ref{T: general leray}. Moreover, by the trivial relation $\hat\Omega\ll \Omega$, we only need to prove $D_t\ll \hat \Omega_t$ for a.e. $t\in [0,T]$. As already argued in the proof of Theorem \ref{T: measure vort}, it suffices to prove 
    \begin{equation}
      \label{D ll hat omega for pos times}
      D_t\ll \hat \Omega_t \qquad \text{for a.e. } t\in [\delta,T],
  \end{equation}
  for any $\delta>0$.  Without loss of generality we can assume  $\nu|\omega^\nu|^2\overset{*}{\rightharpoonup} \tilde D$ in $\mathcal{M}(\T^2\times [0,T])$. We will prove that, for any $\delta>0$, it holds
\begin{equation}
      \label{tilde D ll hat omega for pos times}
      \tilde D\ll \hat \Omega \qquad \text{as measures on } \T^2\times  [\delta,T].
  \end{equation}
Then \eqref{D ll hat omega for pos times} directly follows by Proposition \ref{P: D and tilde D are equiv} together with Lemma \ref{L:slice_measures}. Since from now on $\delta>0$ will be fixed, we will not keep track of it in all the estimates below. Most of them degenerate as $\delta\rightarrow 0$.

Let $\alpha>0$ and denote by $\omega^\nu_\alpha:=\omega^\nu *\rho_\alpha$ the space mollification of $\omega^\nu$. Let $\varphi\in C^0(\T^2\times [0,T])$ be an arbitrary non-negative test function such that $\varphi\leq 1$. By localizing \eqref{fundamental split dissipation} on $\varphi$ we get
\begin{align}\label{use of fundamental split}
    \nu\int_\delta^T \int_{\T^2} |\omega^\nu|^2\varphi \leq \nu \int_\delta^T\int_{\T^2} \omega^\nu\omega^\nu_\alpha \varphi +C\frac{\alpha}{\sqrt{\nu}}.
\end{align}
Moreover
\begin{align}
|\omega^\nu(x,t)||\omega^\nu_\alpha (x,t)|&\leq |\omega^\nu(x,t)|\int |\omega^\nu(y,t)| \rho_\alpha(x-y)\,dy\\
&\leq \frac{C}{\alpha^2}|\omega^\nu(x,t)|\int_{B_\alpha(x)} |\omega^\nu(y,t)|\,dy.
\end{align}
Let $\eps\in (0,1)$ be arbitrary. By plugging this last estimate into \eqref{use of fundamental split} and choosing $\alpha=\sqrt{\nu \eps}$ we obtain
\begin{equation}\label{where eps less 1 used}
    \nu\int_\delta^T \int_{\T^2} |\omega^\nu|^2\varphi \leq C \left(\frac{1}{\eps} \int_\delta^T\int_{\T^2} \hat \Omega^\nu\varphi +C\sqrt{\eps}\right),
\end{equation}
where $\hat \Omega^\nu$ is the function defined in \eqref{modified vort}.
Thus, by letting $\nu\rightarrow 0$, we get 
$$
\int_\delta^T \int_{\T^2} \varphi\, d\tilde D \leq C \left(\frac{1}{\eps} \int_\delta^T\int_{\T^2} \varphi \,d\hat \Omega +C\sqrt{\eps}\right),
$$
valid for all continuous $0\leq \varphi\leq 1$, for a constant $C>0$ independent on $\varphi$ and $\eps$. This yields to  
\begin{equation}\label{nice hat omega bound for small eps}
\tilde D (A)\leq C\left( \frac{\hat \Omega (A)}{\eps}+\sqrt{\eps}\right) \qquad \forall A\subset \T^2\times [\delta,T],\, A \text{ Borel},
\end{equation}
from which \eqref{tilde D ll hat omega for pos times} immediately follows since $\eps>0$ was arbitrary.
\end{proof}

\begin{remark}
    As for Remark \ref{R:D vs Lambda quantitative}, the absolute continuity $D\ll \Omega$ can be made quantitative for positive times. Indeed, a direct consequence of \eqref{use of fundamental split} applied with $\alpha=\sqrt{\eps \nu}$, together with \eqref{P: D and tilde D are equiv}, is the following estimate 
    $$
    D(A)\leq C_\delta\left(\frac{ \Omega (A)}{\eps}+\sqrt{\eps}\right) \qquad \forall A\subset \T^2\times [\delta,T],\, A \text{ Borel},
    $$
    valid for any\footnote{This is in contrast to \eqref{nice hat omega bound for small eps} which holds for $\eps\in (0,1)$ only. The restriction $\eps<1$ has indeed been used to derive \eqref{where eps less 1 used}.} $\eps,\delta>0$. By choosing $\eps=\Omega^{\sfrac{2}{3}}(A)$ we obtain $D(A)\leq C_\delta \Omega^{\sfrac{1}{3}}(A)$ for any Borel set $A\subset \T^2\times [\delta,T]$.
\end{remark}

\begin{proof}[Proof of Corollary \ref{C: pos vort and atomic}]
     We prove the two claims separately. 

     \underline{\textsc{Proof of $(a)$}}. Let $\{\hat \Omega^\nu\}_\nu\subset L^\infty([0,T];L^1(\T^2))$ be the sequence defined in \eqref{modified vort}. Since $\omega^\nu_0=f^\nu_0+\mu^\nu_0$ with $\mu_0^\nu\geq 0$ and $\{f^\nu_0\}_\nu\subset L^1(\T^2)$ relatively compact, by Proposition \ref{P: vort on balls Delort} and \eqref{vortic stays a measure} we get 
     \begin{align}
     \int_{\T^2} \hat \Omega^\nu(x,t)\,dx&\leq \|\omega^\nu(t)\|_{L^1_x}\sup_{x\in \T^2} \int_{B_{\sqrt{\nu}}(x)} |\omega^\nu(y,t)|\,dy\\
     &\leq C \|\omega_0^\nu\|_{\mathcal M_x}\left(G_\beta(\sqrt{\nu} )+\frac{1}{\sqrt{\log\frac{1}{\nu}}}\right),
      \end{align}
      for all sufficiently small $\nu>0$, where $G_\beta$ is the function defined in Definition \ref{D:inverse_beta} (see also Proposition \ref{P:inverse beta well defined}). This shows $\hat \Omega^\nu \rightarrow 0$ in $L^\infty([0,T];L^1(\T^2))$ and by Theorem \ref{T: measure vort} we conclude $D=0$.

      \underline{\textsc{Proof of $(b)$}}. We are assuming that 
      \begin{equation}\label{gamma is product assumpt}
          |\omega^\nu|\otimes |\omega^\nu|\overset{*}{\rightharpoonup}\Gamma \quad \text{in } L^\infty([0,T];\mathcal M(\T^2\times \T^2)),\quad \text{with} \quad \Gamma_t=\gamma_t\otimes \gamma_t \quad \text{for a.e. } t\in [0,T],
      \end{equation}
      for some $\gamma\in L^\infty([0,T];\mathcal M (\T^2))$. Let 
      $$
      \mathscr G_t:=\left\{ x\in \T^2\, : \,  \gamma_t(\{x\})>0\right\}.
      $$
      Since $\gamma_t$ is a finite measure for a.e. $t$, $\mathscr G_t$ is at most countable. The goal is to show that 
      \begin{equation}
          \label{hat omega atomic}
          \hat \Omega_t = \hat \Omega_t \llcorner \mathscr G_t \qquad \text{for a.e. } t\in [0,T].
      \end{equation}
      Since by Theorem \ref{T: measure vort} we know $D_t\ll \hat \Omega_t$, the validity of \eqref{hat omega atomic} forces $D_t$ to be purely atomic for a.e. $t\in [0,T]$. Then $D_t=D_t\llcorner \left(\mathscr L_t \cap \mathscr O_t\right)$ is a direct consequence of $D_t\ll \Lambda_t$ (proved in Theorem \ref{T: general leray}) and $D_t\ll \Omega_t$ (proved in Theorem \ref{T: measure vort}).

      We are left to prove \eqref{hat omega atomic}. Let $\varphi \in C^0(\T^2\times [0,T])$ be an arbitrary non-negative function and fix $r>0$. For any $0<\sqrt{\nu}<r$ we estimate 
      \begin{align}
          \int_0^T\int_{\T^2} \varphi (x,t) \hat \Omega^\nu(x,t)\,dxdt&\leq  \int_0^T\int_{\T^2} \int_{B_r(x)}\varphi (x,t)|\omega^\nu(x,t)| |\omega^\nu(y,t)|\,dydxdt\\
          &\leq \int_0^T\int_{\T^2} \int\varphi (x,t) \chi_r(y-x)|\omega^\nu(x,t)| |\omega^\nu(y,t)|\,dydxdt,
      \end{align}
      where $\chi_r\in C^\infty_c(B_{2r}(0))$ is such that $0\leq \chi_r\leq 1$ and $\chi_r\big|_{B_r(0)}\equiv 1$. Thus, by letting $\nu\rightarrow 0$ we obtain 
      \begin{equation}\label{Gamma on diagonal}
       \int_0^T\int_{\T^2} \varphi \,d\hat \Omega\leq  \int_0^T\left(\int_{\T^2} \int\varphi (x,t) \chi_r(y-x)\,d\Gamma_t (x,y)\right)dt.
    \end{equation}
      By the assumption \eqref{gamma is product assumpt}, this yields to 
        $$
       \int_0^T\int_{\T^2} \varphi \,d\hat \Omega\leq  \int_0^T\left(\int_{\T^2} \varphi (x,t) \gamma_t(B_{2r}(x))\,d\gamma_t(x)\right)dt\qquad \forall r>0.
      $$
      Since $\gamma_t(B_{2r}(x))\rightarrow \gamma_t(\{x\})$ as $r\rightarrow 0$, for all $x\in \T^2$ and a.e. $t\in [0,T]$, by the Lebesgue dominated convergence theorem we deduce\footnote{Note that $x\mapsto \gamma_t(\{x\})$ is, for a.e. $t$, an everywhere defined Borel map.} 
        \begin{equation}\label{preparing concusion atomic}
       \int_0^T\int_{\T^2} \varphi \,d\hat \Omega\leq  \int_0^T\left(\int_{\T^2} \varphi (x,t) \gamma_t(\{x\})\,d\gamma_t(x)\right)dt\qquad \forall\varphi \in C^0(\T^2\times [0,T]).
          \end{equation}
      The measure $\gamma_t(\{x\})d\gamma_t$ is purely atomic and concentrated on $\mathscr G_t$ for a.e. $t$, the atoms of $\gamma_t$. Then \eqref{preparing concusion atomic} becomes 
 \begin{equation}
       \int_0^T\int_{\T^2} \varphi \,d\hat \Omega\leq  \int_0^T\sum_{x\in \mathscr G_t}\varphi (x,t) \gamma_t^2(\{x\})\, dt\qquad \forall\varphi \in C^0(\T^2\times [0,T]),
          \end{equation}
          from which \eqref{hat omega atomic}  follows.
\end{proof}

\begin{remark}\label{R:relax product assumption}
    By \eqref{Gamma on diagonal} it is clear that $\hat \Omega_t$, and thus $D_t$ too, must be purely atomic as soon as $\Gamma_t$ is a discrete measure when restricted to the diagonal $\{ x \times x\, :\, x\in \T^2\}\subset \T^2\times \T^2$. This slightly relaxes the assumption $\Gamma_t=\gamma_t\otimes \gamma_t$.
    \end{remark}

\begin{remark}\label{R:schocet examp}
    In \cite{scho95}*{Pg 1102} the author provides a smooth sequence of vorticities $\{\omega^\nu\}_\nu$, bounded in $\Lip([0,T];W^{-2,1}(\T^2))\cap L^\infty([0,T];H^{-1}(\T^2) ) \cap L^\infty([0,T];\mathcal M(\T^2))$, such that $|\omega^\nu|\otimes |\omega^\nu|\overset{*}{\rightharpoonup} \Gamma$ in $L^\infty([0,T];\mathcal M(\T^2\times \T^2))$ for some $\Gamma$ characterized as
    $$
    \int_0^T\left(\int_{\T^2\times \T^2}\varphi (x,y,t)\,d\Gamma_t(x,y)\right)dt=\frac{1}{2\pi}\int_0^T \int_{-\pi}^\pi \varphi \left(\begin{pmatrix} 1 \\0 \end{pmatrix} \sin \theta, \begin{pmatrix} 1 \\0 \end{pmatrix} \sin \theta, t\right)\, d\theta dt
    $$
    for all continuous $\varphi$.  In particular\footnote{By Fubini's theorem, any measure of the form $\gamma\otimes \gamma$ on a product space must be discrete when restricted to the diagonal.}, $\nexists \gamma_t\in\mathcal M(\T^2)$ such that $\Gamma_t =\gamma_t\otimes \gamma_t$ for a.e. $t$. Consequently, the convergence $|\omega^\nu_t|\overset{*}{\rightharpoonup}\Omega_t$ in $\mathcal M(\T^2)$ cannot hold almost everywhere in time. This shows that the known uniform bounds for \eqref{NS} do not suffice to show that $\Gamma_t$ is a product measure, and the pure atomicity of $D_t$ might be, in principle, ruined by wild oscillations in time. 
\end{remark}

\begin{remark}
    \label{R:no need of compactness initial}
      As it is clear from the proof of Theorem \ref{T: general leray}, the assumption $\{u^\nu_0\}_\nu\subset L^2(\T^2)$ bounded is enough to get $D\in L^\infty_{\rm loc}((0,T];\mathcal M(\T^2))$. In particular all the absolute continuities and concentrations of $D_t$, for a.e. $t$, proved in Theorem \ref{T: general leray}, Theorem \ref{T: measure vort} and Corollary \ref{C: pos vort and atomic} remain true even without the strong compactness  at the initial time. However, in this case it is not possible anymore to deduce that $D=0$ as a space-time measure by $D_t=0$ for a.e. $t$, since $D$ might concentrate some mass at the initial time (see Remark \ref{R:D initial time concentration} below). As we have seen in Proposition \ref{P:dissipation short times quantitative}, this pathological behavior is ruled out if $\{u^\nu_0\}_\nu\subset L^2(\T^2)$ is strongly compact, which then allows to fully characterize $D$ by only looking at almost all time slices.
\end{remark}

\begin{remark}
    \label{R:D initial time concentration}
    We give an example on the whole space $\R^2$. For a given radial and average-free $\omega_0\in C^\infty_c(\R^2)$, solve $\partial_t \omega =\Delta \omega$. Then 
    $$
    \omega^\nu(x,t):=\frac{1}{\nu^2}\omega\left(\frac{x}{\nu},\frac{t}{\nu}\right)
    $$
    solves $\partial_t \omega^\nu =\nu \Delta \omega^\nu$ with initial data  $\omega_0^\nu (x) :=\frac{1}{\nu^2}\omega_0\left(\frac{x}{\nu}\right)$. Since radially symmetric, this defines a solution to \eqref{NS} as well. Moreover $\|\omega^\nu(t)\|_{L^1_x}\leq \|\omega_0\|_{L^1_x}$ for all $t,\nu>0$ and
    $$
    \nu\int_0^\nu \|\omega^\nu(t)\|_{L_x^2}^2\,dt= \int_0^1 \|\omega(t)\|^2_{L^2_x}\,dt\qquad \text{for all } \nu>0.
    $$
In particular it holds $\liminf_{\delta\rightarrow 0} D(\R^2\times [0,\delta])>0$. It can be checked that the corresponding initial velocities are compactly supported, stay bounded in $L^2(\R^2)$, but fail to converge strongly. Note also that $\{u^\nu\}_\nu$ is bounded\footnote{In fact, it goes strongly to zero in $L^2([0,T];H^\alpha(\R^2))$ for all $\alpha<\frac12$.} in $L^2([0,T];H^\frac12(\R^2))$. Thus, the strong $L^2(\R^2\times [0,T])$ compactness of the sequence of velocities is not enough to prevent the energy dissipation to concentrate at the initial time.
\end{remark}

\begin{remark}\label{R:sharp atomic dynamics}
By giving up on the strong compactness of the initial data, it is easy to construct an example on $\R^2$ where all the measures have an atom, for all times. Indeed, let $\omega_0\in C^\infty_c(\R^2)$ be a non-trivial radial profile with zero average. On $\R^2\times (0,\infty)$, solve $\partial_t\omega=\Delta \omega$ with initial condition $\omega_0$. Then, for $\nu>0$, the function
$$
\omega^\nu(x,t):=\frac{1}{\nu}\omega \left(\frac{x}{\sqrt{\nu}},t\right)
$$
solves $\partial_t\omega^\nu=\nu\Delta \omega^\nu$ with the, appropriately rescaled, corresponding initial condition. By the radial symmetry and the zero average condition this defines a sequence of solutions to \eqref{NS} with the corresponding velocities $\{u^\nu\}_\nu$ bounded in $L^\infty([0,\infty);L^2(\R^2))$. Direct computations show that all the measures $\Lambda_t, D_t$ and $\Omega_t$ have an atom at the origin, for all times $t\geq 0$. It follows that the argument used to prove $(b)$ in Corollary \ref{C: pos vort and atomic} is sharp. Indeed, the strong compactness of the initial velocities has only been used to rule out the atomic concentration of $D$ at time $t=0$.
\end{remark}

\section{The dissipative scale}\label{S:K41 scale}
We recall that, in the notation from \eqref{ell compact uniform}, we have
\begin{align}
    S^\nu_2(\ell):= \int_0^T \fint_{B_{\ell}(0)}\mathcal S^\nu_2(y,t) \, dydt\qquad  \text{with} \qquad 
     \mathcal S^\nu_2(y,t):=\norm{u^\nu (\cdot + y, t)  - u^\nu(\cdot, t)}_{L^2_x}^2. 
\end{align}
Theorem \ref{T: K41 scale intro} follows by the following quantitative bounds. 

\begin{proposition}
    \label{P:k41 scale quantitative bounds}
     Let  $\{u_0^\nu\}_{\nu}\subset L^2(\T^2)$ be a bounded sequence of divergence-free vector fields and let $\{u^\nu\}_{\nu}$ be the corresponding sequence of Leray--Hopf solutions to \eqref{NS}. Then
    \begin{equation}
        \label{k41 comp bounded by AD}
        S^\nu_{2}(\sqrt \nu)\leq  \nu \int_0^T \|\nabla u^\nu (t)\|^2_{L^2_x}\,dt
    \end{equation}
    and there exists a constant $C>0$ such that
    \begin{equation}
  \label{Ad bounded by k41 comp}
       \nu \int_\delta^T \|\nabla u^\nu (t)\|^2_{L^2_x}\,dt\leq  \frac{C}{\sqrt{\delta}} \Big(  S^\nu_2 (\sqrt\nu )\Big)^\frac12
\end{equation}
for all $\delta,\nu\in (0,1)$.
\end{proposition}
\begin{proof}
    We have
    \begin{equation}\label{trivial bound compact and dissip}
    \int_0^T \mathcal S^\nu_2(y,t)\, dt\leq |y|^2 \int_0^T \|\nabla u^\nu (t)\|^2_{L^2_x}\,dt\qquad \forall y,
       \end{equation}
       from which the bound \eqref{k41 comp bounded by AD} follows by
       $$
       \fint_{B_{\sqrt\nu}(0)}\int_0^T \mathcal S^\nu_2(y,t) \, dtdy\leq \sup_{|y|\leq \sqrt \nu}  \int_0^T \mathcal S^\nu_2(y,t)\, dt\leq \nu\int_0^T \|\nabla u^\nu (t)\|^2_{L^2_x}\,dt.
       $$
    We are left to prove \eqref{Ad bounded by k41 comp}. Denote by $u^\nu_\alpha:=u^\nu*\rho_\alpha$ the space mollification of $u^\nu$. We split
    \begin{align}
        \nu\int_\delta^T\int_{\T^2}  |\nabla u^\nu|^2 
        &= -\underbrace{\nu \int_\delta^T \int_{\T^2} (u^\nu-u^\nu_\alpha) \cdot \Delta u^\nu }_{I_{\nu,\alpha}}  +\underbrace{\nu  \int_\delta^T\int_{\T^2} \nabla u^\nu_\alpha :\nabla u^\nu}_{II_{\nu,\alpha}}.
    \end{align}
    A direct computation shows 
    \begin{equation}
    \label{mollfied bounds}
    \norm{(u^\nu-u^\nu_\alpha)(t)}^2_{L^2_x} +\alpha^2\norm{\nabla u^\nu_\alpha (t)}^2_{L^2_x}\leq C \fint_{B_{\alpha}(0)} \mathcal S^\nu_2(y,t)\, dy,
    \end{equation}
    for some constant $C>0$ independent on $\alpha,\nu$ and $t$. Thus, by the Cauchy--Schwarz inequality and \eqref{higher order dynamical} we get 
\begin{align}
    \left| I_{\nu,\alpha}\right| &\leq \nu \left(\int_\delta^T \|\Delta u^\nu(t)\|^2_{L^2_x}\, dt\right)^\frac12\left(\int_\delta^T  \norm{(u^\nu-u^\nu_\alpha)(t)}^2_{L^2_x}\,dt\right)^\frac12\\
    &\leq  \frac{C}{\sqrt{\delta}} \left(\int_\delta^T \fint_{B_{\alpha}(0)} \mathcal S^\nu_2(y,t)\, dydt\right)^\frac12.
\end{align}
Similarly, by the energy balance \eqref{NS en bal} and \eqref{mollfied bounds}, we deduce 
\begin{align}
    \left| II_{\nu,\alpha}\right| &\leq \nu \left(\int_\delta^T \|\nabla u^\nu(t)\|^2_{L^2_x}\, dt\right)^\frac12\left(\int_\delta^T \|\nabla u^\nu_\alpha(t)\|^2_{L^2_x}\, dt\right)^\frac12\\
    &\leq C \frac{\sqrt{\nu}}{\alpha} \left(\int_\delta^T \fint_{B_{\alpha}(0)} \mathcal S^\nu_2(y,t)\, dydt\right)^\frac12.
\end{align}
Thus, the choice $\alpha=\sqrt{\nu}$ leads to \eqref{Ad bounded by k41 comp}.
\end{proof}

\begin{remark}
    \label{R:K41 compactness and no AD localized}
    Essentially by the same proof, the following   local version of \eqref{Ad bounded by k41 comp} can be obtained
    $$
    \nu\int_\delta^T \int_{\T^2} |\nabla u^\nu|^2\varphi \leq C_\delta\left(\int_\delta^T \fint_{B_{\sqrt\nu}(0)} \int_{\spt \varphi(\cdot,t)} |u^\nu(x+y,t)-u^\nu(x,t)|^2\,dxdydt\right)^\frac12 +O(\sqrt{\nu}),
    $$
    for all $\varphi\in C^\infty(\T^2\times [0,T])$.
\end{remark}

\begin{remark}
  As soon as the initial data are bounded in $L^2(\T^2)$, a direct consequence of \eqref{trivial bound compact and dissip} is
  $$
  \lim_{\nu\rightarrow 0}\frac{\ell_\nu}{\sqrt{\nu}} =0 \qquad \Longrightarrow \qquad \lim_{\nu\rightarrow 0}  S^\nu_{2}(\ell_\nu)=0.
  $$
  Thus the velocity field always retains compactness strictly inside the dissipative range.
\end{remark}

\begin{proof}[Proof of Theorem \ref{T: K41 scale intro}]
    The right-to-left implication in \eqref{k41 comp implies no AD} is a direct consequence of \eqref{k41 comp bounded by AD}. We are left to prove the left-to-right one. Let $\eps>0$. Since we are assuming $\{u^\nu\}_\nu\subset L^2(\T^2)$ to be strongly compact, by \eqref{diss short time equicont} we find $\delta>0$ such that 
    $$
    \limsup_{\nu\rightarrow 0}\nu \int_0^\delta \|\nabla u^\nu(t)\|^2_{L^2_x}\,dt<\eps.
    $$
    Thus
    $$
     \limsup_{\nu\rightarrow 0}\nu \int_0^T \|\nabla u^\nu(t)\|^2_{L^2_x}\,dt < \eps + \limsup_{\nu\rightarrow 0}\nu \int_\delta^T \|\nabla u^\nu(t)\|^2_{L^2_x}\,dt.
    $$
    In particular, if $ S^\nu_2(\sqrt \nu)\rightarrow 0$, by \eqref{Ad bounded by k41 comp} we deduce that the very last term in the above inequality vanishes. It follows 
    $$
     \limsup_{\nu\rightarrow 0}\nu \int_0^T \|\nabla u^\nu(t)\|^2_{L^2_x}\,dt < \eps,
    $$
   which concludes the proof by the arbitrariness of $\eps>0$.
\end{proof}

\begin{remark}
    \label{R:diss scale asymptotic}
    As it is clear from the proof, the antecedent in the left-to-right implication in \eqref{k41 comp implies no AD} can be relaxed by assuming
    $$
    \lim_{\nu\rightarrow 0}S^\nu_2(\ell_\nu) = 0 \qquad \text{for some } \{\ell_\nu\}_\nu \text{ s.t. } \limsup_{\nu\rightarrow 0}\frac{\sqrt{\nu}}{\ell_\nu}<\infty,
    $$
    while the consequence in the right-to-left can strengthen to 
    $$
    \lim_{\nu\rightarrow 0}S^\nu_{2}(\ell_\nu) = 0 \qquad \text{for all } \{\ell_\nu\}_\nu \text{ s.t. } \limsup_{\nu\rightarrow 0}\frac{\ell_\nu}{\sqrt{\nu}}<\infty.
    $$
\end{remark}

When the initial vorticity is a measure, it is possible to obtain the \quotes{concentration} counterparts of \eqref{Ad bounded by k41 comp} in terms of both velocity and vorticity. We recall the definitions of the main objects
\begin{align}
\Lambda^{\nu}_{\rm con} (\ell) &:= \int_0^T \left(\sup_{x\in \T^2}\int_{B_{\ell}(x)} | u^\nu(y,t)-u (y,t)|^2\,dy \right)^\frac12 dt,\\
 Q^\nu_{\rm con}(\ell)&:=\int_0^T \left(\sup_{x\in \T^2}\int_{B_{\ell}(x)} \left| u^\nu(y,t)-\fint_{B_{\ell}(x)} u^\nu (z,t) \,dz\right|^2\,dy \right)^\frac12 dt,\\
    \Omega^{\nu}_{\rm con} (\ell)&:= \int_0^T \left(\sup_{x\in \T^2}\int_{B_{\ell}(x)} |\omega^\nu (y,t)|\,dy \right)dt. 
\end{align}
\begin{proposition}
    \label{P:k41 concentration quantitative bounds}
    Let  $\{u_0^\nu\}_{\nu}\subset L^2(\T^2)$ be a bounded sequence of divergence-free vector fields such that $\{\omega^\nu_0\}_\nu \subset \mathcal M(\T^2)$ is  bounded. Let $\{u^\nu\}_{\nu}$ be the corresponding sequence of Leray--Hopf solutions to \eqref{NS}. There exists a constant $C>0$ such that
    \begin{align}
        \nu \int_\delta^T\|\omega^\nu(t)\|^2_{L^2_x}\,dt&\leq C\left(\frac{1}{\eps}\Lambda_{\rm con}^\nu(\sqrt \nu) +\int_0^T\left(\sup_{x\in \T^2}\int_{B_{\sqrt{\nu}}(x)} |u(y,t)|^2\,dy\right)^\frac12dt+\sqrt{\frac{\eps}{\delta}}\right),\label{AD bounded by concentration of veloc} \\
        \nu \int_\delta^T\|\omega^\nu(t)\|^2_{L^2_x}\,dt  &\leq C\left(\frac{1}{\eps}Q_{\rm con}^\nu(\sqrt \nu) +\sqrt{\frac{\eps}{\delta}}\right)\label{very ugly estimate bmo}
    \end{align}
    and
     \begin{equation}
        \label{AD bounded by concentration of vortic}
        \nu \int_\delta^T\|\omega^\nu(t)\|^2_{L^2_x}\,dt\leq C\left(\frac{1}{\eps}\Omega_{\rm con}^\nu(\sqrt \nu)+\sqrt{\frac{\eps}{\delta}}\right),
    \end{equation}
    for all $\eps,\delta,\nu \in (0,1)$. Conversely, there exists a constant $C>0$ such that
    \begin{align}\label{AD bound bmo con}
    Q^\nu_{\rm con}(\sqrt \nu)\leq C \left(\nu \int_0^T\|\nabla u^\nu(t)\|^2_{L^2_x}\,dt\right)^\frac12
    \end{align}
    and
    \begin{align}\label{AD bound vort con}
    \Omega^\nu_{\rm con}(\sqrt \nu)\leq C \left(\nu \int_0^T\|\nabla u^\nu(t)\|^2_{L^2_x}\,dt\right)^\frac12,
    \end{align}
    for all $\nu\in (0,1)$.
\end{proposition}
\begin{proof}
    Denote by $\omega^\nu_\alpha:=\omega^\nu*\rho_\alpha$ the space mollification of $\omega^\nu$. By \eqref{fundamental split dissipation} we estimate 
    \begin{align}
         \nu \int_\delta^T\|\omega^\nu(t)\|^2_{L^2_x}\,dt&\leq C\left( \nu \int_\delta^T \int_{\T^2}\omega^\nu\omega_\alpha^\nu +\frac{\alpha}{\sqrt{\nu\delta}}\right)\\
         &\leq C\left( \nu \int_\delta^T \|\omega^\nu(t)\|_{L^1_x}\|\omega_\alpha^\nu(t)\|_{L^\infty_x}\,dt +\frac{\alpha}{\sqrt{\nu\delta}}\right)\\
          &\leq C\left( \nu \int_\delta^T \|\omega_\alpha^\nu(t)\|_{L^\infty_x}\,dt +\frac{\alpha}{\sqrt{\nu\delta}}\right), \label{a bit ugly bound}
    \end{align}
    where to obtain the last inequality we have used \eqref{vortic stays a measure}. We need to estimate $\|\omega_\alpha^\nu(t)\|_{L^\infty_x}$. This can be done in different ways, leading to \eqref{AD bounded by concentration of veloc}, \eqref{very ugly estimate bmo} and \eqref{AD bounded by concentration of vortic} respectively. 

 We start by 
\begin{align}
\|\omega_\alpha^\nu(t)\|_{L^\infty_x}&=\sup_{x\in\T^2}\left| \int\omega^\nu(y,t) \rho_\alpha(x-y)\,dy\right|\\
&=\sup_{x\in\T^2}\left| \int u^\nu(y,t) \cdot \nabla^\perp\rho_\alpha(x-y)\,dy\right|\\
&\leq \frac{C}{\alpha^3}\sup_{x\in\T^2} \int_{B_\alpha(x)} |u^\nu(y,t)|\,dy\\
&\leq \frac{C}{\alpha^2}\left(\sup_{x\in\T^2} \int_{B_\alpha(x)} |u^\nu(y,t)|^2\,dy\right)^\frac12\\
&\leq \frac{C}{\alpha^2}\sup_{x\in\T^2} \left( \int_{B_\alpha(x)} |u^\nu(y,t) - u(y,t)|^2\,dy+ \int_{B_\alpha(x)} |u(y,t)|^2\,dy\right)^\frac12 \label{L infty est on mollif vort}
\end{align}
    By plugging this estimate into \eqref{a bit ugly bound}, the bound \eqref{AD bounded by concentration of veloc} follows by choosing $\alpha=\sqrt{ \eps \nu}$ for an arbitrary $\eps\in (0,1)$.
    
    A second choice is 
    \begin{align}
\|\omega_\alpha^\nu(t)\|_{L^\infty_x}&=\sup_{x\in\T^2}\left| \int u^\nu(y,t) \cdot \nabla^\perp\rho_\alpha(x-y)\,dy\right|\\
&=\sup_{x\in\T^2}\left| \int \left(u^\nu(y,t) - \fint_{B_{\sqrt{\nu}}(x)} u^\nu (z,t) \,dz\right) \cdot \nabla^\perp\rho_\alpha(x-y)\,dy\right|\\
&\leq \frac{C}{\alpha^2}\left(\sup_{x\in\T^2} \int_{B_\alpha(x)} \left|u^\nu(y,t)- \fint_{B_{\sqrt{\nu}}(x)} u^\nu (z,t) \,dz\right|^2\,dy\right)^\frac12.
    \end{align}
    Then \eqref{very ugly estimate bmo} follows by choosing $\alpha=\sqrt{ \eps \nu}$ for an arbitrary $\eps\in (0,1)$.

   The third and final choice is to bound it as 
 \begin{align}
\|\omega_\alpha^\nu(t)\|_{L^\infty_x}&\leq \sup_{x\in\T^2} \int|\omega^\nu(y,t)| \rho_\alpha(x-y)\,dy\leq \frac{C}{\alpha^2}\sup_{x\in\T^2} \int_{B_\alpha(x)} |\omega^\nu(y,t)|\,dy,\label{sup norm mollified vort}
    \end{align}
    from which \eqref{AD bounded by concentration of vortic} follows by choosing $\alpha=\sqrt{ \eps \nu}$ again.

    By the Poincaré inequality inequality we get 
$$
\int_{B_{\ell}(x)} \left| u^\nu(y,t)-\fint_{B_{\ell}(x)} u^\nu (z,t) \,dz\right|^2\,dy\leq C \ell^2 \int_{B_\ell (x)}|\nabla u^\nu(y,t)|^2\,dy.
$$
Then
$$
Q^\nu_{\rm con}(\ell)\leq C \left(\ell^2 \int_0^T\|\nabla u^\nu(t)\|_{L^2_x}^2\,dt\right)^\frac12\qquad \forall \ell>0,
$$
from which \eqref{AD bound bmo con} follows by choosing $\ell=\sqrt{\nu}$.
    
    By using twice the Cauchy--Schwarz inequality we get
    \begin{align}
        \Omega^\nu_{\rm con} (\ell)&=\int_0^T \left(\sup_{x\in \T^2}\int_{B_{\ell}(x)} |\omega^\nu (y,t)|\,dy \right)dt\leq C \ell \int_0^T \|\omega^\nu(t)\|_{L^2_x}\, dt\\
        &\leq C \left(\ell^2 \int_0^T \|\omega^\nu(t)\|^2_{L^2_x}\, dt\right)^\frac12=C \left(\ell^2 \int_0^T \|\nabla u^\nu(t)\|^2_{L^2_x}\, dt\right)^\frac12.
    \end{align}
    Then, the choice $\ell=\sqrt{\nu}$ proves \eqref{AD bound vort con}.
\end{proof}

\begin{remark}
    \label{R:K41 concentration and no AD localized}
    By essentially following the same proof, \eqref{AD bounded by concentration of veloc}, \eqref{very ugly estimate bmo} and \eqref{AD bounded by concentration of vortic} can be localized as
    \begin{align}
    \nu\int_\delta^T \int_{\T^2} |\omega^\nu|^2\varphi &\leq C\left(\frac{1}{\eps}\int_\delta^T \left( \sup_{x\in \spt \varphi (\cdot,t)}\int_{B_{\sqrt{\eps\nu}}(x)} |u^\nu(y,t)|^2\,dy \right)^\frac12 dt +\sqrt{\frac{\eps}{\delta}}\right)\\
    \nu\int_\delta^T \int_{\T^2} |\omega^\nu|^2\varphi &\leq C\left(\frac{1}{\eps}\int_\delta^T  \left( \sup_{x\in \spt \varphi (\cdot,t)}\int_{B_{\sqrt{\eps\nu}}(x)} \left|u^\nu(y,t)- \fint_{B_{\sqrt{\nu}}(x)} u^\nu (z,t) \,dz\right|^2\,dy \right)^\frac12 dt +\sqrt{\frac{\eps}{\delta}}\right)     \end{align}
    and
    $$
     \nu\int_\delta^T \int_{\T^2} |\omega^\nu|^2\varphi \leq C\left(\frac{1}{\eps}\int_\delta^T \left( \sup_{x\in \spt \varphi (\cdot,t)}\int_{B_{\sqrt{\eps\nu}}(x)} |\omega^\nu(y,t)|\,dy \right) dt +\sqrt{\frac{\eps}{\delta}}\right),
    $$
    for all $\varphi\in C^\infty(\T^2\times [0,T])$ and all $\eps,\delta,\nu>0$.
\end{remark}
We can now prove Theorem \ref{T: k41 scale concentration} and Theorem \ref{T:new full charact at dissipative scale}.
\begin{proof}[Proof of Theorem \ref{T: k41 scale concentration}]
Let $\eps>0$. Since we are assuming $\{u^\nu_0\}_\nu\subset L^2(\T^2)$ to be strongly compact, by \eqref{diss short time equicont} we find $\delta>0$ such that 
    $$
    \limsup_{\nu\rightarrow 0}\nu \int_0^\delta \|\nabla u^\nu(t)\|^2_{L^2_x}\,dt<\eps.
    $$
    Thus
    \begin{equation}\label{usual boring split}
     \limsup_{\nu\rightarrow 0}\nu \int_0^T \|\nabla u^\nu(t)\|^2_{L^2_x}\,dt < \eps + \limsup_{\nu\rightarrow 0}\nu \int_\delta^T \|\nabla u^\nu(t)\|^2_{L^2_x}\,dt.
       \end{equation}
       Let $\tilde \eps\in(0,1)$. By \eqref{AD bounded by concentration of veloc} we get 
       \begin{align}
            \nu \int_\delta^T\|\nabla u^\nu(t)\|^2_{L^2_x}\,dt&\leq C\left(\frac{1}{\tilde \eps}\Lambda_{\rm con}^\nu(\sqrt \nu) +\int_0^T\left(\sup_{x\in \T^2}\int_{B_{\sqrt{\nu}}(x)} |u(y,t)|^2\,dy\right)^\frac12dt+\sqrt{\frac{\tilde \eps}{\delta}}\right).
       \end{align}
       Since $u\in L^\infty([0,T];L^2(\T^2))$, by the dominated convergence theorem and the absolute continuity of the Lebesgue integral, we deduce 
       $$
       \lim_{\nu\rightarrow 0}\int_0^T\left(\sup_{x\in \T^2}\int_{B_{\sqrt{\nu}}(x)} |u(y,t)|^2\,dy\right)^\frac12dt=0.
       $$
       Together with the assumption $ \Lambda^{\nu}_{\rm con} (\sqrt \nu)\rightarrow 0$, this yields to 
       $$
       \limsup_{\nu\rightarrow 0}  \nu \int_\delta^T\|\nabla u^\nu(t)\|^2_{L^2_x}\,dt\leq C \sqrt{\frac{\tilde \eps}{\delta}}. 
       $$
       Then \eqref{usual boring split} becomes 
       $$
        \limsup_{\nu\rightarrow 0}  \nu \int_0^T\|\nabla u^\nu(t)\|^2_{L^2_x}\,dt<\eps+C\sqrt{\frac{\tilde \eps}{\delta}}.
       $$
       Note that $\delta$ does not depend on $\tilde \eps$. Thus, we can first send $\tilde\eps\rightarrow 0$ and then $\eps\rightarrow 0$, concluding the proof of \eqref{condition on veloc L2}. 
       
       The left-to-right implication in \eqref{condition on vort L1} follows by the very same argument given above by using \eqref{AD bounded by concentration of vortic} instead of \eqref{AD bounded by concentration of veloc}, while the right-to-left implication in \eqref{condition on vort L1} is a direct consequence of \eqref{AD bound vort con}. Details are left to the reader.
\end{proof}

\begin{remark}
    \label{R:diss scale asymptotic vort}
    The right-to-left implication in \eqref{condition on vort L1} can be strengthen to
    $$
    \lim_{\nu\rightarrow 0}\nu \int_0^T \|\nabla u^\nu(t)\|^2_{L^2_x}\,dt=0\quad \Longrightarrow \quad \lim_{\nu\rightarrow 0} \Omega^\nu_{\rm con} (\ell_\nu)=0\quad \text{for all } \{\ell_\nu\}_\nu \text{ s.t. } \limsup_{\nu\rightarrow 0}\frac{\ell_\nu}{\sqrt{\nu}}<\infty.
    $$
\end{remark}
\begin{proof}[Proof of Theorem \ref{T:new full charact at dissipative scale}]
The second equivalence in \eqref{new full equivalence equation} has been already proved in Theorem \ref{T: k41 scale concentration}. We are left to show 
\begin{equation}
\label{Q and no AD equivalence proof}
\lim_{\nu\rightarrow 0} Q^\nu_{\rm con}(\sqrt{\nu}) =0  \qquad \Longleftrightarrow \qquad \lim_{\nu\rightarrow 0} \nu \int_0^T \|\nabla u^\nu (t)\|^2_{L^2_x}\,dt=0.
\end{equation}
The right-to-left implication in \eqref{Q and no AD equivalence proof} follows from \eqref{AD bound bmo con}. As already done several times, in view of Proposition \ref{P:dissipation short times quantitative}, to obtain the converse implication it is enough to prove 
\begin{equation}
    \label{Q and no AD positive times}
    \lim_{\nu\rightarrow 0} Q^\nu_{\rm con}(\sqrt{\nu}) =0  \qquad \Longrightarrow \qquad \lim_{\nu\rightarrow 0} \nu \int_\delta^T \|\omega^\nu (t)\|^2_{L^2_x}\,dt=0\quad \forall\delta >0.
\end{equation}
This is a direct consequence of \eqref{very ugly estimate bmo}, concluding the proof.
\end{proof}

\begin{remark}\label{R:Q bmo refined scales}
    As for Remark \ref{R:diss scale asymptotic vort}, the right-to-left implication in the first equivalence in \eqref{new full equivalence equation} can be refined as 

    $$
    \lim_{\nu\rightarrow 0}\nu \int_0^T \|\nabla u^\nu(t)\|^2_{L^2_x}\,dt=0\quad \Longrightarrow \quad \lim_{\nu\rightarrow 0} Q^\nu_{\rm con} (\ell_\nu)=0\quad \text{for all } \{\ell_\nu\}_\nu \text{ s.t. } \limsup_{\nu\rightarrow 0}\frac{\ell_\nu}{\sqrt{\nu}}<\infty.
    $$
\end{remark}

\begin{remark}
As it is clear from the proofs, all the implications 
$$
 \lim_{\nu\rightarrow 0}\nu \int_0^T \|\nabla u^\nu(t)\|^2_{L^2_x}\,dt=0\qquad \Longrightarrow \qquad \left\{\begin{array}{ll}
\lim_{\nu\rightarrow 0}  S^\nu_{2}(\sqrt{\nu}) =0 \\ 
\lim_{\nu\rightarrow 0} \Omega^\nu_{\rm con}(\sqrt{\nu}) =0\\
\lim_{\nu\rightarrow 0} Q^\nu_{\rm con}(\sqrt{\nu}) =0
\end{array}\right.
$$
hold for any sequence of vector fields $\{u^\nu\}_\nu$, thus independently on any uniform regularity and any PDE. On the other hand, the reverse implications rely on \eqref{NS}.
\end{remark}

\section{Quantitative rates and dissipation life-span}\label{S:rates}
By Lemma \ref{L:beta_int_on_ball} and Proposition \ref{P:split_vort_L1_pos} it is possible to quantify the uniform in time vorticity decay on balls  whenever the singular part is non-negative.
\begin{proposition}
    \label{P: vort on balls Delort}
    Let $\{u_0^\nu\}_{\nu}\subset L^2(\T^2)$ be a bounded sequence of divergence-free vector fields such that  $\{\omega^\nu_0\}_{\nu}\subset \mathcal M(\T^2)$ admits a decomposition $\omega_0^\nu =f^\nu_0 + \mu_0^\nu$ with $\mu^\nu_0\geq 0$ and   $\left\{ f^\nu_0\right\}_{\nu}\subset L^1(\T^2)$ satisfying 
$$
\sup_{\nu>0}\int_{\T^2} \beta\left(|f^\nu_0(x)| \right)\,dx<\infty
$$    
for some $\beta\in \mathcal K$, the set defined in \eqref{def superlinear beta}.
Denote by 
     $$
 M:=\sup_{\nu>0 }\left( \|u^\nu_0\|_{L^2_x}+ \int_{\T^2} \beta\left(|f^\nu_0(x)| \right)\,dx\right).
 $$
Let $G_\beta$ be the function given by Definition \ref{D:inverse_beta}. There exists a constant $C>0$ depending only on $M$ and a value $r_0\in (0,1)$ depending only on $\beta$ such that the sequence of vorticities $\{\omega^\nu\}_{\nu}$ of the corresponding Leray--Hopf solutions satisfies
\begin{equation}
   \sup_{x,t,\nu} \int_{B_r(x)} |\omega^\nu(y,t)|\,dy\leq C\left( G_\beta(r) + \frac{1}{\sqrt{\log \frac{1}{r}}}\right)\qquad \forall0< r<r_0.
\end{equation}
\end{proposition}
\begin{proof}
    Let $x\in \T^2$. We define the cut-off function 
    \begin{align}
\chi_r(y):=
\begin{cases}
 1 & \text{ if $y\in B_r(x)$}\\
  \frac{\log \frac{|y-x|}{\sqrt{r}}}{\log \sqrt{r}}& \text{ if $y\in B_{\sqrt{r}}(x)\setminus B_r(x)$}\\
0 &\text{ if $y\in B_{\sqrt{r}}^c(x)$.}
\end{cases}
\end{align}
A direct computation shows
$$
\int |\nabla \chi_r(y)|^2\,dy=\frac{1}{|\log \sqrt r|^2} \int_{ B_{\sqrt{r}}(x)\setminus B_r(x)}\frac{1}{|y|^2}\,dy\leq \frac{C}{\log \frac{1}{r}}\qquad \forall 0<r<\frac12.
$$
Thus, by Proposition \ref{P:split_vort_L1_pos} and Lemma \ref{L:beta_int_on_ball} we conclude
\begin{align}
    \int_{B_r(x)} |\omega^\nu(y,t)|\,dy&\leq \int  |\omega^\nu(y,t)|\chi_r(y)\,dy\\
    &\leq 2 \int  |f^\nu(y,t)|\chi_r(y)\,dy + \int  \omega^\nu(y,t)\chi_r(y)\,dx\\
    &\leq 2 \int_{B_{\sqrt{r}}(x)}  |f^\nu(y,t)|\,dy - \int u^\nu(y,t)\cdot \nabla^\perp \chi_r(y)\,dy\\
    &\leq C G_\beta(r) + \|u^\nu(t)\|_{L^2_x} \|\nabla \chi_r\|_{L_x^2}\\
    &\leq C\left( G_\beta(r) + \frac{1}{\sqrt{\log \frac{1}{r}}}\right) \qquad\forall 0<r<\min\left(r_0,\frac{1}{2}\right),
\end{align}
where $r_0>0$ is the value given by Lemma \ref{L:beta_int_on_ball}. Note that to obtain the last inequality we have also used $\|u^\nu(t)\|_{L^2_x}\leq \|u^\nu_0\|_{L^2_x}$, together with the assumption $\{u^\nu_0\}_\nu\subset L^2(\T^2)$ bounded.
\end{proof}

We can now prove Theorem \ref{T: rates intro}.
\begin{proof}[Proof of Theorem \ref{T: rates intro}]
Let $\alpha>0$. By using \eqref{vortic stays a measure} and \eqref{sup norm mollified vort} into \eqref{refined fundamental split dissipation} we get 
$$
  \nu\int_\delta^T\|\omega^\nu(t)\|^2_{L^2_x}\,dt\leq C\left(  \frac{ \nu T}{\alpha^2} \sup_{x,t,\nu}\int_{B_\alpha(x)}  |\omega^\nu(y,t)| \, dy + \frac{\alpha^2}{\nu\delta} \right),
$$
for some $C>0$ depending only on $M_1$, the constant defined in the statement of Theorem \ref{T: rates intro}. Denote by $\tilde G_\beta(s):=G_\beta(s)+\left(\log\frac{1}{s}\right)^{-\sfrac{1}{2}}$ for convenience. By Proposition \ref{P: vort on balls Delort} there exists $\alpha_0$ depending only on $\beta$ such that 
\begin{equation}
    \label{quant bound ugly}
     \nu\int_\delta^T\|\omega^\nu(t)\|^2_{L^2_x}\,dt\leq C\left(  \frac{ \nu T}{\alpha^2} \tilde G_\beta(\alpha) + \frac{\alpha^2}{\nu\delta} \right)\qquad \forall 0<\alpha<\alpha_0,
\end{equation}
where the constant $C>0$ now depends on $M_1$ and $M_2$. Set $\nu_0:=\alpha_0^2$ and let $\eps\in (0,1)$ be arbitrary. By choosing $\alpha=\sqrt{\eps \nu}$ in \eqref{quant bound ugly} we get\footnote{Note that $\tilde G_\beta(\sqrt{\eps \nu})\leq \tilde G_\beta(\sqrt{\nu})$ since $\tilde G_\beta$ is monotone non-decreasing and $\eps\in (0,1)$.}
$$
  \nu\int_\delta^T\|\omega^\nu(t)\|^2_{L^2_x}\,dt\leq C\left(  \frac{T}{\eps} \tilde G_\beta(\sqrt{\nu}) + \frac{\eps}{\delta} \right)\qquad \forall 0<\nu<\nu_0.
$$
We wish to choose $\eps=\sqrt{\delta T \tilde G_\beta(\sqrt{\nu})}$ as it optimizes the above inequality. Since we required $\eps<1$, this choice is certainly possible if \eqref{restriction on T and nu} holds. Thus, if \eqref{restriction on T and nu} is satisfied,  we achieved
$$
\nu\int_\delta^T\|\omega^\nu(t)\|^2_{L^2_x}\,dt\leq C\sqrt{\frac{T}{\delta}\tilde G_\beta(\sqrt\nu)} \qquad \forall 0<\nu<\nu_0.
$$
This proves \eqref{diss_rate}.

Assume now $\{u^\nu_0\}_\nu\subset L^2(\T^2)$ to be strongly compact. Let $\eps>0$. By \eqref{diss short time equicont} we find $\delta\in(0,1)$ such that 
\begin{equation}
    \label{usual bound short times}
    \limsup_{\nu\rightarrow 0}\nu \int_0^\delta\|\nabla u^\nu(t)\|^2_{L^2_x}\,dt<\eps.
\end{equation}
Note that, if $\{T_\nu\}_\nu$ is a sequence of positive real numbers satisfying \eqref{times assump}, we can find $\nu_1>0$ such that \eqref{restriction on T and nu} holds for all $0<\nu<\nu_1$ and with $T_\nu$ in place of $T$. Thus, by \eqref{diss_rate} we get
\begin{equation}
    \label{bound long times}
    \nu\int_\delta^{T_\nu}\|\nabla u^\nu(t)\|^2_{L^2_x}\,dt\leq C \sqrt{\frac{T_\nu}{\delta}\tilde G_\beta(\sqrt\nu)} \qquad \forall 0<\nu<\min(\nu_0,\nu_1).
\end{equation}
By putting together \eqref{usual bound short times}, \eqref{bound long times} and the assumption \eqref{times assump} we conclude
$$
\limsup_{\nu\rightarrow 0}\nu \int_0^{T_\nu}\|\nabla u^\nu(t)\|^2_{L^2_x}\,dt<\eps+ \limsup_{\nu\rightarrow 0}  \nu\int_\delta^{T_\nu}\|\nabla u^\nu(t)\|^2_{L^2_x}\,dt=\eps,
$$
from which \eqref{diss_vanish_long_times} follows by the arbitrariness of $\eps>0$.
 \end{proof}

\begin{remark}\label{R:time scale 1 over nu}
Let us show that in a time scale $T_\nu\gtrsim \nu^{-1}$ the dissipation is always non-trivial. Assume $\int_{\T^2}u^\nu_0=0$. Since the zero average condition is preserved along the evolution, by the energy balance \eqref{NS en bal} and the Poincar\'e inequality 
$$
\frac{d}{dt}\|u^\nu(t)\|^2_{L^2_x}=-2\nu \|\nabla u^\nu(t)\|^2_{L^2_x}\leq -\nu C \|u^\nu(t)\|^2_{L^2_x},
$$
from which $\|u^\nu(t)\|^2_{L^2_x}\leq \|u^\nu_0\|^2_{L^2_x} e^{-\nu Ct}$ by the Gr\"onwall lemma. Thus, if $T_\nu\geq (\nu C)^{-1}$ we deduce 
$$
2\nu\int_0^{T_\nu}\|\nabla u^\nu(t)\|^2_{L^2_x}\,dt=\|u^\nu_0\|^2_{L^2_x}-\|u^\nu(T_\nu)\|^2_{L^2_x}\geq \frac12 \|u^\nu_0\|^2_{L^2_x}.
$$
This proves that, as soon as the initial data do not converge strongly to zero, it must hold
$$
T_\nu\geq\frac{1}{\nu C}\qquad \Longrightarrow \qquad \liminf_{\nu\rightarrow 0}\nu\int_0^{T_\nu}\|\nabla u^\nu(t)\|^2_{L^2_x}\,dt>0.
$$
\end{remark}

 By Proposition \ref{P:dissipation short times quantitative} it follows that any quantitative compactness of the initial data allows to get a  rate for the dissipation up to the initial time, thus extending \eqref{diss_rate} all the way to $\delta=0$. While several choices are possible, we give a particular example in the next proposition. We emphasize that here the final time $T$ is fixed a priori.
\begin{proposition}
    \label{P: diss rates up to zero}
 Let $\{u^\nu_0\}_\nu\subset L^2(\T^2)$ be a bounded sequence of divergence-free vector fields such that 
 \begin{equation}
     \label{averaged Besov uniform}
     \sup_{\nu>0}\left(\fint_{B_\ell(0)} \|u^\nu_0(\cdot+y)-u^\nu_0(\cdot)\|^2_{L^2_x}\,dy\right)^\frac12\leq C\ell^\sigma\qquad \forall\ell>0,
 \end{equation}
 for some $C,\sigma>0$. Assume that $\{\omega^\nu_0\}_\nu\subset \mathcal{M}(\T^2)$ is bounded and it admits a decomposition $\omega^\nu_0=f^\nu_0+\mu^\nu_0$  for some  $\mu^\nu_0\geq 0$ and $\{f^\nu_0\}_{\nu}\subset L^1(\T^2)$ such that 
\begin{equation}
    \sup_{\nu>0} \int_{\T^2} \beta \left(|f_0^\nu (x)|\right)\,dx<\infty,
\end{equation}
for some $\beta\in \mathcal K$, the set defined in \eqref{def superlinear beta}. There exist a constant $C>0$ and a value $\nu_0>0$ such that 
$$
\nu\int_0^T \|\nabla u^\nu(t)\|^2_{L^2_x}\,dt\leq C\left(G_\beta(\sqrt\nu)+\frac{1}{\sqrt{\log\frac{1}{\nu}}}\right)^{\frac{\sigma}{4(1+\sigma)}}\qquad \forall 0<\nu<\nu_0. 
$$
\end{proposition}
\begin{proof}
      Let $\Phi(\eps)$ be as in \eqref{modulus of comapctness}. A direct computation shows that \eqref{averaged Besov uniform} implies $\Phi(\eps)\leq C\eps^\sigma$. Let $\nu_0$ be small enough, depending on $\beta$ and $T$, such that \eqref{diss_rate} holds. By Proposition \ref{P:dissipation short times quantitative} and Theorem \ref{T: rates intro} we deduce 
      $$
      \nu\int_0^T \|\nabla u^\nu(t)\|^2_{L^2_x}\,dt\leq C\left(\sqrt{\eps^\sigma+ \frac{\delta}{\eps^2}}+\sqrt{\frac{\tilde G_\beta(\sqrt\nu)}{\delta}}\right) \qquad \forall 0<\nu<\nu_0,\, \forall \eps,\delta\in (0,1),
      $$
      with $\tilde G_\beta(s):=G_\beta(s)+\left(\log\frac{1}{s}\right)^{-\sfrac{1}{2}}$. Choosing $\delta=\eps^{2+\sigma}$ yields to 
      $$
      \nu\int_0^T \|\nabla u^\nu(t)\|^2_{L^2_x}\,dt\leq C\left(\eps^{\frac{\sigma}{2}}+\frac{\sqrt{\tilde G_\beta(\sqrt\nu)}}{\eps^{1+\frac{\sigma}{2}}}\right) \qquad \forall 0<\nu<\nu_0,\, \forall \eps\in (0,1),
      $$
      which is optimized by $\eps^{1+\sigma}:=\sqrt{\tilde G_\beta(\sqrt\nu)}$. This concludes the proof.
\end{proof}
\begin{remark}\label{R:besov equivalence}
     The assumption \eqref{averaged Besov uniform} is equivalent to ask that $\{u^\nu_0\}_\nu\subset B^\sigma_{2,\infty}(\T^2)$ is bounded\footnote{Recall that $f\in B^\sigma_{2,\infty}$ if $f\in L^2$ and $\|f(\cdot+y)-f(\cdot)\|_{L^2}\leq C |y|^\sigma$ for all $y$.}. Indeed,  denoting by $u^\nu_{0,\ell}=u^\nu_0*\rho_\ell$ its mollification, \eqref{averaged Besov uniform} implies 
     $$
    \sup_{\nu>0} \|u^\nu_{0,\ell}-u^\nu_{0}\|_{L^2_x}\leq C\ell^\sigma \qquad \text{and}\qquad \sup_{\nu>0} \|\nabla u^\nu_{0,\ell}\|_{L^2_x}\leq C\ell^{\sigma-1},
     $$
     from which the, uniform in viscosity, Besov regularity immediately follows, that is,
     \[
     \eqref{averaged Besov uniform} \qquad \implies\qquad  \sup_{\nu>0}\rVert u_0^\nu(\cdot + y)-u_0^\nu(\cdot)\rVert_{L_x^2}\le C|y|^{\sigma}.
     \] The opposite direction is trivial.
     Further conditions in terms of the vorticity decay on balls can be found in \cite{lant23}*{Section 2}.
\end{remark}

\section{The stationary case}\label{S:stationary}
Here we address the case of steady fluids. In this setting, the energy identity for \eqref{SNS} implies
\begin{equation}
    \label{en identity steady}
    \nu \int_{\T^2}|\nabla u^\nu|^2=\int_{\T^2}u^\nu \cdot f^\nu\leq \|u^\nu\|_{L^2} \|f^\nu\|_{L^2} .
\end{equation}
By taking the $\curl$ of the first equation in \eqref{SNS}, the vorticity $\omega^\nu$ solves
\begin{equation}
    \label{vort eq steady}
   u^\nu\cdot \nabla \omega^\nu  =\nu \Delta \omega^\nu +\curl f^\nu.
\end{equation}
In particular
\begin{equation}\label{enstrophy balance steady}
    \nu \int_{\T^2}|\nabla \omega^\nu|^2=\int_{\T^2} \omega^\nu \curl f^\nu=- \int_{\T^2} f^\nu\cdot \nabla^\perp \omega^\nu\leq \norm{f^\nu}_{L^2}\norm{\nabla \omega^\nu}_{L^2},
\end{equation}
and consequently 
\begin{equation}
    \label{higer order estimate steady}
    \nu \norm{\nabla \omega^\nu}_{L^2}\leq \norm{f^\nu}_{L^2}.
\end{equation}
\begin{proof}[Proof of Theorem \ref{T: measure vort steady}]
The structure of the proof follows closely that of Theorem \ref{T: general leray} and Theorem \ref{T: measure vort}, the main difference being how the external force is handled. We break the proof down into steps. 

 \underline{\textsc{Proof of $D\ll\Lambda$}}. Let $\varphi \in C^\infty(\T^2)$. Integrating by parts we split
 \begin{align}
        \nu\int_{\T^2}  |\nabla u^\nu|^2 \varphi
        &= -\underbrace{\nu  \int_{\T^2} (u^\nu-u) \cdot \Delta u^\nu \varphi}_{I_\nu}  -\underbrace{\nu  \int_{\T^2} u \cdot \Delta u^\nu \varphi}_{II_\nu} + \underbrace{\nu  \int_{\T^2} \frac{|u^\nu|^2}{2} \Delta \varphi}_{III_\nu}.\label{eq_splitting_new_proof_lambda}
    \end{align}
Since $\{u^\nu\}_\nu\subset L^2(\T^2)$ is bounded, $III_\nu\rightarrow 0$.  We now handle $II_\nu$. For any $\psi\in C^\infty(\T^2)$ we have 
$$
\nu\int_{\T^2}\psi \cdot \Delta u^\nu = \nu\int_{\T^2}u^\nu \cdot \Delta \psi\rightarrow 0.
$$
This, together with $\{\nu\Delta u^\nu\}_\nu\subset L^2(\T^2)$ bounded thanks to \eqref{higer order estimate steady}, yields to $\nu\Delta u^\nu\rightharpoonup 0$ in $L^2(\T^2)$ and consequently $II_\nu\rightarrow 0$.

We are only left to handle the term $I_\nu$, that is the only non-vanishing one. By \eqref{higer order estimate steady} it can be estimated as
$$
|I_\nu|\leq C \|(u^\nu-u)\varphi\|_{L^2}.
$$
Summing up, since $\nu|\nabla u^\nu|^2\overset{*}{\rightharpoonup} D$ and $|u^\nu-u|^2\overset{*}{\rightharpoonup}\Lambda$, we have proved
\begin{equation}\label{D and Lambda quantitative steady}
\int_{\T^2} \varphi \, dD\leq C \left(\int_{\T^2} \varphi^2 \, d\Lambda\right)^\frac12,
\end{equation}
from which $D\ll \Lambda$ follows by the arbitrariness of $\varphi$.

 \underline{\textsc{Proof of $F=0 \, \Longrightarrow \, D=0$}}. Denote by $\omega_\alpha:=\omega^\nu*\rho_\alpha$ the mollification of $\omega^\nu$. By \eqref{moll est 2}, \eqref{moll est 3} and \eqref{higer order estimate steady} we bound 
 \begin{align}
 \nu\int_{\T^2} |\omega^\nu|^2\leq 2\nu \left( \int_{\T^2}|\omega^\nu_\alpha|^2  +\int_{\T^2}|\omega^\nu-\omega^\nu_\alpha|^2 \right)\leq C \nu \left( \frac{1}{\alpha^2} +\alpha^2 \int_{\T^2} |\nabla \omega^\nu|^2\right).\label{split diss for forcing global steady}
 \end{align}
Denote $f^\nu_\eps:=f^\nu*\rho_\eps$ and define $\Phi(\eps):=\sup_{\nu>0} \|f^\nu_\eps-f^\nu\|_{L^2}$. Note that $\Phi(\eps)\rightarrow 0$ since $F=0$. By manipulating \eqref{enstrophy balance steady} we get 
$$
\nu\int_{\T^2} |\nabla \omega^\nu|^2=\int_{\T^2} \omega^\nu \curl f^\nu =-\int_{\T^2} f^\nu\cdot \nabla^\perp \omega^\nu  \leq \Phi(\eps)\|\nabla \omega^\nu\|_{L^2}+ \int_{\T^2} \omega^\nu \curl f^\nu_\eps.
$$
Thus, by \eqref{higer order estimate steady}, \eqref{en identity steady} and \eqref{moll est 2} we deduce
$$
\nu\int_{\T^2} |\nabla \omega^\nu|^2\leq C\left( \frac{\Phi(\eps)}{\nu}+\|\omega^\nu\|_{L^2}\frac{\|f^\nu\|_{L^2}}{\eps}\right)\leq C\left( \frac{\Phi(\eps)}{\nu}+\frac{1}{\eps \sqrt{\nu}}\right).
$$
By plugging this last estimate into \eqref{split diss for forcing global steady} we achieve
$$
\nu\int_{\T^2} |\omega^\nu|^2\leq C \left( \frac{\nu}{\alpha^2} +\frac{\alpha^2}{\nu} \left(\Phi(\eps) +\frac{\sqrt\nu}{\eps}\right)\right)\qquad \forall\nu,\alpha,\eps>0.
$$
Consequently, we choose\footnote{A more optimal choice can be made of course.} $\eps:=\nu^{\sfrac{1}{4}}$ and then $\alpha^2:=\nu \left(\Phi(\nu^{\sfrac{1}{4}}) + \nu^{\sfrac{1}{4}}\right)^{-\sfrac{1}{2}}$ to get 
$$
\nu\int_{\T^2} |\omega^\nu|^2\leq C\sqrt{\Phi(\nu^{\sfrac{1}{4}}) + \nu^{\sfrac{1}{4}}},
$$
from which we conclude $D=0$ by letting $\nu\rightarrow 0$.

  \underline{\textsc{Proof of $D=D\llcorner  (\mathscr L \cap \mathscr O$)}}. Since we already proved $D\ll \Lambda$, it is enough to show $D=D\llcorner  \mathscr O$, or equivalently $D(\mathscr O^c)=0$. By possibly passing to a subsequence, we can assume $\nu|\omega^\nu|^2\overset{*}{\rightharpoonup} \tilde D$ in $\mathcal{M}(\T^2)$.   Let $\varphi\in C^\infty(\T^2)$. Denoting by $\omega^\nu_\alpha:=\omega^\nu*\rho_\alpha$ the mollification of $\omega^\nu$, we split 
  \begin{equation}
      \label{tilde D splitting steady}
      \nu\int_{\T^2}|\omega^\nu|^2\varphi =\nu\int_{\T^2}\omega^\nu \omega^\nu_\alpha \varphi +\nu\int_{\T^2}\omega^\nu(\omega^\nu -\omega^\nu_\alpha)\varphi.
  \end{equation}
  By \eqref{moll est 3}, \eqref{en identity steady} and \eqref{higer order estimate steady} we get 
\begin{equation}
    \label{second small concentration steady}
    \nu \left| \int_{\T^2}\omega^\nu(\omega^\nu -\omega^\nu_\alpha)\varphi\right|\leq \nu \alpha\|\omega^\nu\|_{L^2} \|\nabla \omega^\nu\|_{L^2}\|\varphi\|_{L^\infty}\leq C \|\varphi\|_{L^\infty} \frac{\alpha}{\sqrt{\nu}}.
\end{equation}
Moreover
\begin{equation}
    \label{first product measures steady}
    \nu \left| \int_{\T^2}\omega^\nu \omega^\nu_\alpha \varphi\right|\leq \frac{\nu}{\alpha^2} \int_{\T^2} \int |\omega^\nu (x)||\omega^\nu (y)| \rho\left(\frac{x-y}{\alpha}\right)\varphi(x)\, dydx. 
\end{equation}
Let $\eps,r>0$. By choosing $\alpha=\sqrt{\eps \nu}$ and plugging \eqref{second small concentration steady} and \eqref{first product measures steady} into \eqref{tilde D splitting steady} we achieve
 \begin{equation}
      \nu\int_{\T^2}|\omega^\nu|^2\varphi \leq \frac{1}{\eps}\int_{\T^2} \int |\omega^\nu (x)||\omega^\nu (y)| \rho\left(\frac{x-y}{r}\right)\varphi(x)\, dydx +C \sqrt{\eps}  \|\varphi\|_{L^\infty},
  \end{equation}
  as soon as $r\geq \sqrt{\eps\nu}$. Thus, since $\nu|\omega^\nu|^2\overset{*}{\rightharpoonup} \tilde D$ in $\mathcal{M}(\T^2)$ and $|\omega^\nu|\otimes |\omega^\nu|\overset{*}{\rightharpoonup} \Omega\otimes \Omega$ in $\mathcal{M}(\T^2\times \T^2)$, by letting $\nu\rightarrow 0$ we obtain 
  \begin{align}
  \int_{\T^2} \varphi \, d\tilde D&\leq \frac{1}{\eps}\int_{\T^2} \int  \rho\left(\frac{x-y}{r}\right)\varphi(x)\, d(\Omega\otimes \Omega) +C \sqrt{\eps} \|\varphi\|_{L^\infty}\\
 &\leq \frac{C}{\eps}\int_{\T^2}  \varphi(x)\Omega (B_r(x)) \, d\Omega (x) +C \sqrt{\eps} \|\varphi\|_{L^\infty}\qquad \forall r>0.
\end{align}
Clearly $\Omega(B_r(x))\rightarrow \Omega (\{x\} )$ for all $x\in \T^2$. Thus, by letting $r\rightarrow 0$ we achieve 
\begin{equation}\label{nice bound steady}
\int_{\T^2} \varphi \, d\tilde D \leq \frac{C}{\eps}\int_{\T^2} \int  \varphi(x)\Omega (\{x\} ) \, d\Omega (x) +C \sqrt{\eps} \|\varphi\|_{L^\infty}.
\end{equation}
The measure $\Omega \left(\{x\} \right) d\Omega$ is finite and purely atomic, concentrated on $\mathscr O$. Then \eqref{nice bound steady} becomes 
\begin{equation}\label{nice bound steady atomic}
\int_{\T^2} \varphi \, d\tilde D \leq \frac{C}{\eps}\sum_{x\in \mathscr O}  \varphi(x)\Omega^2 (\{x\} )  +C \sqrt{\eps} \|\varphi\|_{L^\infty}.
\end{equation}
Since $D\leq C \tilde D$ as measures (see Proposition \ref{P: D and tilde D are equiv}), the arbitrariness of $\varphi$ then implies
\begin{equation}\label{inequality measures final steady}
D(A)\leq C \tilde D(A)\leq C\left( \frac{1}{\eps} \sum_{x\in \mathscr O\cap A}  \Omega^2 (\{x\} )  + \sqrt{\eps}\right)\qquad \forall A\subset \T^2, \, A \text{ Borel}.
\end{equation}
Thus $D(\mathscr O^c)\leq C \sqrt{\eps}$,
which yields to $D (\mathscr O^c)=0$ since $\eps>0$ was arbitrary.
\end{proof}

\begin{remark}
    In fact, from \eqref{D and Lambda quantitative steady} we get $D(A)\leq C\Lambda^{\sfrac{1}{2}}(A)$ for any Borel set $A\subset \T^2$, i.e. the absolute continuity is quantitative.  Also, if we enumerate $\mathscr O=\{x_i\}_i$, by setting $D_i:=D(\{x_i\})$, and $\Omega_i:=\Omega(\{x_i\})$, an optimization in $\eps$ of the inequality \eqref{inequality measures final steady} yields to $D_i\leq C \Omega_i^{\sfrac{2}{3}}$ for all $i$.
\end{remark}

 \begin{remark}
     Without the vorticity being a measure, there is no hope to constraint the dissipation to be purely atomic, nor lower dimensional. The vector field $u^\nu(x_1,x_2):=\sin \left( \frac{x_2}{\sqrt{\nu}}\right) e_1$ solves \eqref{SNS} with $f^\nu=u^\nu$ and $p^\nu=0$. Clearly $\{u^\nu\}_\nu\subset L^\infty (\T^2)$ is bounded, it converges weakly to $0$ in $L^2(\T^2)$, but not strongly. Moreover $\nu|\nabla u^\nu|^2=\left|\cos \left(\frac{x_2}{\sqrt \nu}\right)\right|^2$. In particular, all the measures $\Lambda, F$ and $D$ are non-trivial and absolutely continuous with respect to the Lebesgue measure.
 \end{remark}

\begin{remark}
    \label{R:D ll F fails}
The proof of $F=0\,\Longrightarrow \, D=0$  does not seem to be improvable to $D\ll F$. The main obstruction is the appearance of problematic (cubic) terms when trying to localize \eqref{split diss for forcing global steady} on a test function $\varphi$. These terms do not even seem to stay bounded as $\nu\rightarrow 0$. Perhaps, since by Theorem \ref{T: measure vort steady} we have $D\ll \Lambda$, an attempt would be to deduce $D\ll F$ from $\Lambda\ll F$. However, the latter fails in general. Indeed, consider the stream function $\psi^\nu(x_1,x_2):=\nu^{-\kappa}\sin (\nu^\kappa x_1) \cos (\nu^\kappa x_2)$, for some $\kappa\in \left(-\frac12, 0\right)$. Then $\omega^\nu=\Delta \psi^\nu =-2\nu^{2\kappa}\psi^\nu$. In particular, if $u^\nu=\nabla^\perp \psi^\nu$, it holds $u^\nu\cdot \nabla \omega^\nu=0$. We have thus obtained a solution to \eqref{SNS} with $f^\nu:=-\nu \Delta u^\nu$. Moreover, $u^\nu\rightharpoonup 0$ in $L^2(\T^2)$, $|u^\nu|^2\overset{*}{\rightharpoonup} c \mathcal{L}^2$ in $\mathcal{M}(\T^2)$ for some $c>0$, while $\|f^\nu\|_{L^2}\leq C\nu^{1+2\kappa}\rightarrow 0$ since $\kappa>-\frac12$.
\end{remark}

\begin{remark}
    As for the time dependent case, the compactness of $\{f^\nu\}_\nu$ at  scales $\ell_\nu\sim \sqrt{\nu}$ suffices to rule out dissipation. Indeed, a closer inspection at the proof given above shows that $D=0$ as soon as
    $$
    \fint_{B_{\ell_\nu}(0)} \|f^\nu(\cdot + y )-f^\nu(\cdot)\|^2_{L^2}\,dy\rightarrow 0
    $$
    for a sequence of positive numbers $\{\ell_\nu\}_\nu$ such that $\limsup_{\nu\rightarrow 0}\frac{\sqrt{\nu}}{\ell_\nu}=0$. A similar consideration applies to $\{u^\nu\}_\nu$.
\end{remark}

\begin{remark}
    \label{R: sharp stationary}
    Any radial vorticity profile $\omega\in C^\infty_c(\R^2)$ defines a stationary solution to the incompressible Euler equations. If $\omega$ has zero average and is compactly supported, then $u$ is also compactly supported in $\mathbb{R}^2$ (see for instance \cite{GSPS21}*{Lemma 4.3}). Being compactly supported, we can also think of it as a solution on $\T^2$. Thus, $u^\nu(x):= \frac{1}{\sqrt\nu} u\left(\frac{x}{\sqrt{\nu}} \right)$, or analogously  $\omega^\nu(x):= \frac{1}{\nu}\omega\left(\frac{x}{\sqrt{\nu}} \right)$, defines a sequence of smooth compactly supported solutions to \eqref{SNS} with $f^\nu=-\nu \Delta u^\nu$. Moreover, if $\omega$ is non-trivial, all the measures $D,\Lambda, F$ and $\Omega$ have an atom at the origin. In particular, no concentration-cancellation can hold.  
\end{remark}

\section{Kinematic examples}\label{S:kin examp}
The next proposition serves to highlight that atoms may independently appear in the defect measure $\Lambda$ and in the vorticity measure $\Omega$. In view of Corollary \ref{C: pos vort and atomic}, this leaves open the possibility to rule out dissipation if the two measures concentrate on disjoint sets.

\begin{proposition}
    \label{P: atoms in lambda vs omega}
    Let $B_1\subset\R^2$ be the open disk of radius $1$ centered at the origin. The following hold.
    \begin{itemize}
        \item[(i)]There exists a sequence $\{u_n\}_n\subset C^\infty_c(B_1)$ of incompressible vector fields such that, denoting by $\omega_n:=\curl u_n$, it holds
 \begin{equation}\label{atoms in lambda and no in omega}
    \|u_n\|_{L^2}\rightarrow 0\qquad \text{ and } \qquad |\omega_n|\overset{*}{\rightharpoonup} \Omega \text{ with } \, \Omega(\{0\})>0.
 \end{equation}
 \item[(ii)] There exists a sequence $\{u_n\}_n\subset C^\infty_c(B_1)$ of incompressible vector fields such that $\|u_n\|_{L^2}=1$ for all $n$ and, denoting by $\omega_n:=\curl u_n$, it holds
 \begin{equation}
     \|u_n\|_{L^1}+ \|\omega_n\|_{L^1}\rightarrow 0 \qquad  \text{ and }\qquad  |u_n|^2\overset{*}{\rightharpoonup} \Lambda  \text{ with } \Lambda(\{0\})>0.
 \end{equation}
    \end{itemize}
    Being compactly supported, the above examples work on $\T^2$ as well.
\end{proposition}
\begin{proof} We prove the two separately.

\underline{\textsc{Proof of $(i)$}}.  We claim that $\exists \{v_n\}_n\subset C^\infty_c(B_1)$ with $\div v_n=0$ for all $n$, such that 
    \begin{equation}\label{claim one}
     \|\curl v_n\|_{L^1} =1 \quad\forall n \qquad \text{ and } \qquad \|v_n\|_{L^2}\rightarrow 0.
     \end{equation}
    Assuming the validity of the claim, the conclusion follows by suitably rescaling the sequence. Indeed, for any $\eps_n\in (0,1)$ such that $\eps_n\rightarrow 0$, we can set 
    $$
    u_n(x):=\frac{1}{\eps_n} v_n \left( \frac{x}{\eps_n}\right) \qquad \text{on } \R^2.
    $$
    Clearly $\spt u_n\subset B_{\eps_n}(0)$ for all $n$. Moreover, $\|u_n\|_{L^2}=\|v_n\|_{L^2}\rightarrow 0$ while $\|\omega_n\|_{L^1}=\|\curl v_n\|_{L^1}=1$ for all $n$. Thus, we may assume $|\omega_n|\overset{*}{\rightharpoonup} \Omega$ in $\mathcal M (B_1)$. Let $\delta\in(0,1)$ be arbitrary. Since $\spt \omega_n\subset B_{\eps_n}(0)$, by the upper semi-continuity of weak* convergence of measures on compact sets we have 
    $$
    \Omega\left(\overline B_\delta(0)\right)\geq \limsup_{n\rightarrow \infty} \int_{B_\delta(0)}|\omega_n |=1.
    $$
    By letting $\delta \rightarrow 0$ we deduce $\Omega(\{0\})=1$.  We are left to prove \eqref{claim one}. Pick any $v\in L^2(B_1)$ with $\spt v\subset B_{1}$, $\div v=0$ and $\curl v\not \in \mathcal M(B_1)$. By mollifying it, we obtain a sequence $\{\tilde v_n\}_n\subset C^\infty_c (B_1)$ such that $\|\tilde v_n\|_{L^2}\leq \|v\|_{L^2}$ and $\|\curl \tilde v_n\|_{L^1}\rightarrow \infty$. Then, the sequence of vector fields 
$$
v_n:=\frac{\tilde v_n}{\|\curl \tilde v_n\|_{L^1}}
$$
has the desired properties. 

    \underline{\textsc{Proof of $(ii)$}}.  We claim that $\exists \{v_n\}_n\subset C^\infty_c(B_1)$ with $\div v_n=0$ for all $n$, such that 
    \begin{equation}\label{interesting claim}
    \|v_n\|_{L^1}+ \|\curl v_n\|_{L^1}\rightarrow 0 \qquad \text{ and } \qquad \|v_n\|_{L^2}=1 \quad \forall n.
    \end{equation}

    As in the above proof, the conclusion then follows by the same rescaling argument. Details are left to the reader. Let us prove \eqref{interesting claim}. Take any compactly supported measure $\tilde \mu$ with an atom at the origin and, uniquely among functions decaying at infinity, solve 
    \begin{equation}\label{Dirac in Laplacian}
    \Delta \tilde \psi = \tilde \mu \qquad \text{on }   \R^2.
\end{equation}
    Fix a cut-off $\chi\in C^\infty_c (B_1)$ such that $\chi\equiv 1$ in a neighbourhood of the origin. Then $\psi:= \chi\tilde \psi$ is compactly supported in $B_1$. The classical Calder\'on--Zygmund theory implies $\tilde \psi \in W^{1,1}(B_1)$. Thus 
    $$
    \Delta \psi=\chi \tilde \mu + 2 \nabla \chi \cdot \nabla \tilde \psi + \tilde \psi \Delta \chi
    $$
    is a compactly supported measure with an atom at the origin. In particular, see for instance \cite{delort1991existence}*{Lemma 1.2.5},  we deduce $\nabla \psi \not \in L^2(B_1)$.
    By mollifying it, we obtain a sequence $\{ \psi_n\}_n\subset C^\infty_c(B_1)$ such that $\| \nabla \psi_n\|_{L^2}\rightarrow \infty$, $\sup_n \| \Delta \psi_n\|_{L^1}<\infty$ and $\sup_n\| \nabla \psi_n\|_{L^1}\leq \| \nabla \psi\|_{L^1}<\infty$. Then, the sequence 
    $$
    v_n:=\frac{\nabla^\perp \psi_n}{\|\nabla^\perp \psi_n\|_{L^2}}
    $$
    has all the desired properties.
\end{proof}

We now turn to the relation between atomic concentrations in the vorticity and strong compactness of the velocity. Consider a sequence $\{u_n\}_n$ of incompressible vector fields. By the Lions concentration compactness principle \cite{Struwe}*{Section 4.8}, whenever $\{\nabla u_n\}_n \subset L^1(\T^2)$ is bounded, $\Lambda$ is a purely atomic measure which can display an atom at a point only if that point is an atom appearing in the weak* limit of $\{|\nabla u_n|\}_n$ in $\mathcal M(\T^2)$. This is due to the Sobolev embedding $W^{1,1}(\T^2)\subset L^2(\T^2)$. However, the failure of the Calder\'on--Zygmund estimate at the endpoint allows the vorticity to be $L^1(\T^2)$  without the corresponding velocity being necessarily $L^2(\T^2)$. In particular, as we shall show in the next proposition, the defect measure might diffuse, thus failing to be lower dimensional, even if the vorticity stays bounded in $L^1(\T^2)$.
\begin{proposition}
    \label{P: no weak lions}
   Let $Q=(-1,1)^2\subset \R^2$. There exists a sequence $\{u_n\}_n\subset C^\infty_c (Q)$ of incompressible vector fields such that, denoting by $\omega_n:=\curl u_n$, it holds 
$$
\|u_n\|_{L^1}\rightarrow 0, \qquad \|\omega_n\|_{L^1}=1\quad \forall n \qquad \text{and}\qquad |u_n|^2\overset{*} {\rightharpoonup} \frac14 \mathcal L^2 \quad \text{in } \mathcal M(Q),
$$
where $\mathcal L^2$ denotes the two-dimensional Lebesgue measure. Being compactly supported, the construction works on $\T^2$ as well.
\end{proposition}
The failure of the lower dimensionality of $\Lambda$ under measure vorticity assumption was already considered  by DiPerna--Majda \cite{diperna1987concentrations}*{pp. 323-325}. This issue lead them to introduce a \quotes{reduced} version of $\Lambda$ in \cite{DM88}, which is lower dimensional in an appropriate sense and allows for concentration-cancellation phenomena (see also \cites{Evans_notes,Lopes97}). As opposite to the DiPerna--Majda constructions, the proof of Proposition \ref{P: no weak lions}  follows by the endpoint failure of Calder\'on--Zygmund only. In particular, the same construction applies to any dimension and any relation between $u$ and $\omega$ with suitable minor modifications, perhaps providing a more robust mechanism. The recent paper \cites{DS24} is also related to this discussion.

\begin{proof}
    Let $\{v_n\}_n \subset C^\infty_c(B_1)$ be the sequence from \eqref{interesting claim}. Denote by $\eps_n:=\|\curl v_n\|_{L^1}$, which vanishes as $n\rightarrow \infty$. For any $n$ divide\footnote{To be precise we should take the integer part of $\eps_n^{-2}$.} $Q$ in $\eps_n^{-2}$ open squares of size $2\eps_n$ and denote by $\{x_{i,n}\}_{i=1}^{\eps_n^{-2}}$ their barycenters. Clearly 
    $$
    x\mapsto v_n\left( \frac{x-x_{i,n}}{\eps_n}\right)
    $$
    is smooth and compactly supported into the $i-$th square. Thus
    $$
    u_n(x):= \sum_{i=1}^{\eps_n^{-2}} v_n\left( \frac{x-x_{i,n}}{\eps_n}\right)
    $$
    defines a sequence of incompressible vector fields  $\{ u_n\}_n\subset C^\infty_c(Q)$. By \eqref{interesting claim} we have 
    $$
    \|u_n\|_{L^1}=\sum_{i=1}^{\eps_n^{-2}} \norm{ v_n\left( \frac{\cdot -x_{i,n}}{\eps_n}\right)}_{L^1}= \|v_n\|_{L^1}\rightarrow 0
    $$
    and 
    $$
    \|\omega_n\|_{L^1}=\frac{1}{\eps_n}\sum_{i=1}^{\eps_n^{-2}} \norm{ (\curl v_n)\left( \frac{\cdot -x_{i,n}}{\eps_n}\right)}_{L^1}= \frac{\|\curl v_n\|_{L^1}}{\eps_n}=1 \qquad \forall n.
    $$
 We are left to prove $ 4|u_n|^2\overset{*} {\rightharpoonup}\mathcal L^2$ in $\mathcal M(Q)$. For any $\varphi \in C^1(\overline Q)$ we split
    \begin{align}
        \int_{Q} |u_n|^2\varphi &=\eps_n^2 \sum_{i=1}^{\eps_n^{-2}}\int_{B_1} |v_n(x)|^2 \varphi (x_{i,n} + \eps_n x)\,dx\\
        &= \eps_n^2 \sum_{i=1}^{\eps_n^{-2}}\int_{B_1} |v_n(x)|^2 \big( \varphi (x_{i,n} + \eps_n x) - \varphi(x_{i,n})\big)\,dx\\
        &+\eps_n^2 \sum_{i=1}^{\eps_n^{-2}}\varphi(x_{i,n}).
    \end{align}
    The first term vanishes since it is bounded by 
    $$
   \eps_n^3  \|\nabla \varphi\|_{L^\infty} \sum_{i=1}^{\eps_n^{-2}}\int_{B_1} |v_n(x)|^2\,dx= \eps_n\|\nabla \varphi\|_{L^\infty} \rightarrow 0,
    $$
    while the second term converges to $ \frac14 \int_{Q} \varphi$.
\end{proof}

\begin{remark}
    By a direct computation it can be proved that the sequence constructed in Proposition \ref{P: no weak lions} satisfies $4|\omega_n|\overset{*}{\rightharpoonup}\mathcal{L}^2$ in $\mathcal M(Q)$.
\end{remark}

\begin{remark}
   The sequence defined in \eqref{interesting claim} could have been chosen to be radially symmetric. In this case, the sequence constructed in Proposition \ref{P: no weak lions} consists of steady solutions to the incompressible Euler equations. In particular, avoiding any explicit construction, our approach reveals the more general mechanism behind \cite{GT88}. We also note that, although the strong compactness fails, concentration compactness  occurs \cites{scho95,delort1991existence} and the weak limit is a weak solution to the stationary Euler equations.
\end{remark}

\bibliographystyle{plain} 
\bibliography{biblio}

@article{V2,
   author={Vishik, M.},
   title={Instability and non-uniqueness in the {C}auchy problem for the {E}uler equations of an ideal incompressible fluid. {P}art {II}},
   note={Preprint available at \href{https://arxiv.org/abs/1805.09440}{arXiv:1805.09440} },
   date={2018},
}

@article {DM88,
    AUTHOR = {DiPerna, R. J. and Majda, A.},
     TITLE = {Reduced {H}ausdorff dimension and concentration-cancellation
              for two-dimensional incompressible flow},
   JOURNAL = {J. Amer. Math. Soc.},
  FJOURNAL = {Journal of the American Mathematical Society},
    VOLUME = {1},
      YEAR = {1988},
    NUMBER = {1},
     PAGES = {59--95},
      ISSN = {0894-0347,1088-6834},
   MRCLASS = {35Q10 (76C99)},
  MRNUMBER = {924702},
MRREVIEWER = {Denis\ Serre},
       DOI = {10.2307/1990967},
       URL = {https://doi.org/10.2307/1990967},
}

@article{brue2023nonuniqueness,
  title={Nonuniqueness of solutions to the {E}uler equations with vorticity in a {L}orentz space},
  author={Bru\'e, E. and Colombo, M.},
  journal={Communications in Mathematical Physics},
  volume={403},
  number={2},
  pages={1171--1192},
  year={2023},
  publisher={Springer}
}

@article{brue2024flexibility,
  title={Flexibility of two-dimensional {E}uler flows with integrable vorticity},
  author={Bru\'e, E. and Colombo, M. and Kumar, A.},
 note={Preprint available at \href{https://arxiv.org/abs/2408.07934}{arXiv:2408.07934} },
  year={2024}
}

@article{buck2024non,
  title={Non-uniqueness and energy dissipation for 2{D} {E}uler equations with vorticity in {H}ardy spaces},
  author={Buck, M. and Modena, S.},
  journal={Journal of Mathematical Fluid Mechanics},
  volume={26},
  number={2},
  pages={26},
  year={2024},
  publisher={Springer}
}

@article{mengual2023dissipative,
  title={Dissipative {E}uler flows for vortex sheet initial data without distinguished sign},
  author={Mengual, F. and Sz{\'e}kelyhidi Jr, L.},
  journal={Comm. on Pure and App. Math.},
  volume={76},
  number={1},
  pages={163--221},
  year={2023},
  publisher={Wiley Online Library}
}

@article{jeong2021vortex,
  title={Vortex stretching and enhanced dissipation for the incompressible 3{D} {N}avier--{S}tokes equations},
  author={Jeong, I.-J. and Yoneda, T.},
  journal={Mathematische Annalen},
  volume={380},
  number={3},
  pages={2041--2072},
  year={2021},
  publisher={Springer}
}

@article{johansson2024anomalous,
  title={Anomalous dissipation via spontaneous stochasticity with a two-dimensional autonomous velocity field},
  author={Johansson, C. J. P. and Sorella, M.},
 note={Preprint available at \href{https://arxiv.org/abs/2409.03599}{arXiv:2409.03599} },
  year={2024},
}

@article{armstrong2025anomalous,
  title={Anomalous diffusion by fractal homogenization},
  author={Armstrong, S. and Vicol, V.},
  journal={Annals of PDE},
  volume={11},
  number={1},
  pages={2},
  year={2025},
  publisher={Springer}
}

@article{crippa2017eulerian,
  title={Eulerian and {L}agrangian Solutions to the Continuity and {E}uler Equations with ${L}^{1}$ Vorticity},
  author={Crippa, G. and Nobili, C. and Seis, C. and Spirito, S.},
  journal={SIAM Journal on Mathematical Analysis},
  volume={49},
  number={5},
  pages={3973--3998},
  year={2017},
  publisher={SIAM}
}

@article{crippa2014renormalized,
  title={Renormalized solutions of the $2${D} Euler equations},
  author={Crippa, G.  and Spirito, S.},
  journal={Comm. Math. Phys.},
  volume={339},
  pages={191–198},
  year={2015},
}

@article{ciampa2021strong,
  title={Strong convergence of the vorticity for the 2{D} {E}uler equations in the inviscid limit},
  author={Ciampa, G. and Crippa, G. and Spirito, S.},
  journal={Archive for Rational Mechanics and Analysis},
  volume={240},
  number={1},
  pages={295--326},
  year={2021},
  publisher={Springer}
}

@article {brue2022onsager,
    AUTHOR = {Bru\'e, E. and Colombo, M. and Crippa, G. and De
              Lellis, C. and Sorella, M.},
     TITLE = {Onsager critical solutions of the forced {N}avier-{S}tokes
              equations},
   JOURNAL = {Comm. Pure Appl. Anal.},
  FJOURNAL = {Communications on Pure and Applied Analysis},
    VOLUME = {23},
      YEAR = {2024},
    NUMBER = {10},
     PAGES = {1350--1366},
      ISSN = {1534-0392,1553-5258},
   MRCLASS = {35Q30 (76D03)},
  MRNUMBER = {4799447},
       DOI = {10.3934/cpaa.2023071},
       URL = {https://doi.org/10.3934/cpaa.2023071},
}

@article{drivas2022anomalous,
  title={Anomalous dissipation in passive scalar transport},
author={Drivas, T. D. and Elgindi, T. M. and Iyer, G. and Jeong, I.-J.},
  journal={Archive for Rational Mechanics and Analysis},
  pages={1--30},
  year={2022},
  publisher={Springer}
}

@article{colombo2023anomalous,
  title={Anomalous dissipation and lack of selection in the {O}bukhov--{C}orrsin theory of scalar turbulence},
  author={Colombo, M. and Crippa, G. and Sorella, M.},
  journal={Annals of PDE},
  volume={9},
  number={2},
  pages={21},
  year={2023},
  publisher={Springer}
}

@article{burczak2023anomalous,
  title={Anomalous dissipation and {E}uler flows},
  author={Burczak, J. and Sz{\'e}kelyhidi Jr, L. and Wu, B.},
 note={Preprint available at \href{https://arxiv.org/abs/2310.02934}{2310.02934} },
  year={2023}
}

@article{brue2023anomalous,
  title={Anomalous dissipation for the forced 3{D} {N}avier--{S}tokes equations},
  author={Bru\'e, E. and De Lellis, C.},
  journal={Comm. in Math. Phys.},
  volume={400},
  number={3},
  pages={1507--1533},
  year={2023},
  publisher={Springer}
}

@article{DS24,
   author={Dom\'inguez, O. and Spector, D.},
   title={A sparse resolution of the {D}i{P}erna-{M}ajda gap problem for $2${D} {E}uler equations},
   note={Preprint available at \href{https://arxiv.org/abs/2409.02344}{2409.02344} },
   date={2024},
}

@article {lant23,
    AUTHOR = {Lanthaler, S.},
     TITLE = {On concentration in vortex sheets},
   JOURNAL = {Partial Differ. Equ. Appl.},
  FJOURNAL = {Partial Differential Equations and Applications},
    VOLUME = {4},
      YEAR = {2023},
    NUMBER = {2},
     PAGES = {Paper No. 13, 39},
      ISSN = {2662-2963,2662-2971},
   MRCLASS = {35Q35 (35Q31 65M12 65M70 76B03)},
  MRNUMBER = {4565116},
       DOI = {10.1007/s42985-023-00230-6},
       URL = {https://doi.org/10.1007/s42985-023-00230-6},
}

@article {DRP24,
    AUTHOR = {De Rosa, L. and Park, J.},
     TITLE = {No anomalous dissipation in two-dimensional incompressible
              fluids},
   JOURNAL = {SIAM J. Math. Anal.},
  FJOURNAL = {SIAM Journal on Mathematical Analysis},
    VOLUME = {57},
      YEAR = {2025},
    NUMBER = {5},
     PAGES = {5771--5790},
      ISSN = {0036-1410,1095-7154},
   MRCLASS = {76D05 (28C05 35D30 35Q30 76F02)},
  MRNUMBER = {4969534},
       DOI = {10.1137/25M1773167},
       URL = {https://doi.org/10.1137/25M1773167},
}

@article{vecchi19931,
  title={On ${L}^1$-vorticity for $2$-{D} incompressible flow},
  author={Vecchi, I. and Wu, S.},
  journal={manuscripta mathematica},
  volume={78},
  pages={403--412},
  year={1993},
  publisher={Springer}
}

@article{Batch69,
  title={Computation of the Energy Spectrum in Homogeneous Two‐Dimensional Turbulence},
  author={Batchelor, G. K.},
  journal={Phys. Fluids},
  volume={12},
  pages={233-239},
  year={1969},
}

@article{Kraich67,
  title={Inertial Ranges in Two‐Dimensional Turbulence },
  author={Kraichnan, R. H.},
  journal={Phys. Fluids},
  volume={10},
  pages={1417-1423},
  year={1967},
}

@article {Hopf51,
    AUTHOR = {Hopf, E.},
     TITLE = {\"{U}ber die {A}nfangswertaufgabe f\"{u}r die hydrodynamischen
              {G}rundgleichungen},
   JOURNAL = {Math. Nachr.},
  FJOURNAL = {Mathematische Nachrichten},
    VOLUME = {4},
      YEAR = {1951},
     PAGES = {213--231},
      ISSN = {0025-584X,1522-2616},
   MRCLASS = {76.1X},
  MRNUMBER = {50423},
MRREVIEWER = {J.\ Kamp\'{e} de F\'{e}riet},
       DOI = {10.1002/mana.3210040121},
       URL = {https://doi.org/10.1002/mana.3210040121},
}

@article {scho95,
    AUTHOR = {Schochet, S.},
     TITLE = {The weak vorticity formulation of the {$2$}-{D} {E}uler
              equations and concentration-cancellation},
   JOURNAL = {Comm. Partial Differential Equations},
  FJOURNAL = {Communications in Partial Differential Equations},
    VOLUME = {20},
      YEAR = {1995},
    NUMBER = {5-6},
     PAGES = {1077--1104},
      ISSN = {0360-5302,1532-4133},
   MRCLASS = {35Q30 (76C99)},
  MRNUMBER = {1326916},
MRREVIEWER = {Denis\ Serre},
       DOI = {10.1080/03605309508821124},
       URL = {https://doi.org/10.1080/03605309508821124},
}

@book {BV22,
    AUTHOR = {Bedrossian, J. and Vicol, V.},
     TITLE = {The mathematical analysis of the incompressible {E}uler and
              {N}avier-{S}tokes equations---an introduction},
    SERIES = {Graduate Studies in Mathematics},
    VOLUME = {225},
 PUBLISHER = {American Mathematical Society, Providence, RI},
      YEAR = {2022},
     PAGES = {xiii+218},
      ISBN = {[9781470470494]},
   MRCLASS = {35-01 (35Q30 35Q31)},
  MRNUMBER = {4475666},
       DOI = {10.1090/gsm/225},
       URL = {https://doi.org/10.1090/gsm/225},
}

@book {K08,
    AUTHOR = {Klenke, A.},
     TITLE = {Probability theory},
    SERIES = {Universitext},
      NOTE = {A comprehensive course,
              Translated from the 2006 German original},
 PUBLISHER = {Springer-Verlag London, Ltd., London},
      YEAR = {2008},
     PAGES = {xii+616},
      ISBN = {978-1-84800-047-6},
   MRCLASS = {60-01 (28-01 60B10 60F05 60F15 60G42 60J10)},
  MRNUMBER = {2372119},
MRREVIEWER = {Sophie\ Lemaire},
       DOI = {10.1007/978-1-84800-048-3},
       URL = {https://doi.org/10.1007/978-1-84800-048-3},
}

@article {CCFS08,
    AUTHOR = {Cheskidov, A. and Constantin, P. and Friedlander, S. and
              Shvydkoy, R.},
     TITLE = {Energy conservation and {O}nsager's conjecture for the {E}uler
              equations},
   JOURNAL = {Nonlinearity},
  FJOURNAL = {Nonlinearity},
    VOLUME = {21},
      YEAR = {2008},
    NUMBER = {6},
     PAGES = {1233--1252},
      ISSN = {0951-7715},
   MRCLASS = {76B03 (76F02)},
  MRNUMBER = {2422377},
MRREVIEWER = {Hee Chul Pak},
       DOI = {10.1088/0951-7715/21/6/005},
       URL = {https://doi.org/10.1088/0951-7715/21/6/005},
}

@article {ELL_tocome,
    AUTHOR = {Elgindi, T. M. and Lopes Filho, M. C. and Nussenzveig
              Lopes, H. J.},
     TITLE = {Absence of anomalous dissipation for vortex sheets},
   JOURNAL = {J. Funct. Anal.},
  FJOURNAL = {Journal of Functional Analysis},
    VOLUME = {290},
      YEAR = {2026},
    NUMBER = {6},
     PAGES = {Paper No. 111304, 25},
      ISSN = {0022-1236,1096-0783},
   MRCLASS = {35Q30 (76D99)},
  MRNUMBER = {5006346},
       DOI = {10.1016/j.jfa.2025.111304},
       URL = {https://doi.org/10.1016/j.jfa.2025.111304},
}

@article {Driv22,
    AUTHOR = {Drivas, T. D.},
     TITLE = {Self-regularization in turbulence from the {K}olmogorov
              4/5-law and alignment},
   JOURNAL = {Philos. Trans. Roy. Soc. A},
  FJOURNAL = {Philosophical Transactions of the Royal Society A.
              Mathematical, Physical and Engineering Sciences},
    VOLUME = {380},
      YEAR = {2022},
    NUMBER = {2226},
     PAGES = {Paper No. 20210033, 15},
      ISSN = {1364-503X,1471-2962},
   MRCLASS = {35Q30},
  MRNUMBER = {4437848},
       DOI = {10.1098/rsta.2021.0033},
       URL = {https://doi.org/10.1098/rsta.2021.0033},
}

@book {Frisch95,
    AUTHOR = {Frisch, U.},
     TITLE = {Turbulence},
      NOTE = {The legacy of A. N. Kolmogorov},
 PUBLISHER = {Cambridge University Press, Cambridge},
      YEAR = {1995},
     PAGES = {xiv+296},
      ISBN = {0-521-45103-5},
   MRCLASS = {76-02 (35Q30 76D05 76Fxx 76M35)},
  MRNUMBER = {1428905},
MRREVIEWER = {Philip J. Holmes},
}

@article {castro2024proof,
    AUTHOR = {Castro, \'A. and Faraco, D. and Mengual, F. and
              Solera, M.},
     TITLE = {A proof of {V}ishik's nonuniqueness theorem for the forced
              2{D} {E}uler equation},
   JOURNAL = {J. Reine Angew. Math.},
  FJOURNAL = {Journal f\"ur die Reine und Angewandte Mathematik. [Crelle's
              Journal]},
    VOLUME = {824},
      YEAR = {2025},
     PAGES = {253--288},
      ISSN = {0075-4102,1435-5345},
   MRCLASS = {35Q31 (35A02)},
  MRNUMBER = {4926947},
       DOI = {10.1515/crelle-2025-0025},
       URL = {https://doi.org/10.1515/crelle-2025-0025},
}

@article {dolce2024self,
    AUTHOR = {Dolce, M. and Mescolini, G.},
     TITLE = {Self-similar instability and forced nonuniqueness: an
              application to the 2{D} {E}uler equations},
   JOURNAL = {J. Lond. Math. Soc. (2)},
  FJOURNAL = {Journal of the London Mathematical Society. Second Series},
    VOLUME = {112},
      YEAR = {2025},
    NUMBER = {2},
     PAGES = {Paper No. e70274, 28},
      ISSN = {0024-6107,1469-7750},
   MRCLASS = {35Q31 (35Q35)},
  MRNUMBER = {4949623},
       DOI = {10.1112/jlms.70274},
       URL = {https://doi.org/10.1112/jlms.70274},
}

@book {de2024instability,
    AUTHOR = {Albritton, D. and Bru\'e, E. and Colombo, M. and De
              Lellis, C. and Giri, V. and Janisch, M. and
              Kwon, H.},
     TITLE = {Instability and non-uniqueness for the 2{D} {E}uler equations,
              after {M}. {V}ishik},
    SERIES = {Annals of Mathematics Studies},
    VOLUME = {219},
 PUBLISHER = {Princeton University Press, Princeton, NJ},
      YEAR = {2024},
     PAGES = {ix+136},
      ISBN = {978-0-691-25752-5; 978-0-691-25753-2},
   MRCLASS = {35-02 (35Q31 76B03 76B47)},
  MRNUMBER = {4729618},
MRREVIEWER = {Francesco\ Fanelli},
}

@article{V1,
   author={Vishik, M.},
   title={Instability and non-uniqueness in the {C}auchy problem for the {E}uler equations of an ideal incompressible fluid. {P}art {I}},
   note={Preprint available at \href{https://arxiv.org/abs/1805.09426}{arXiv:1805.09426} },
   date={2018},
}

@article{majda1993remarks,
  title={Remarks on weak solutions for vortex sheets with a distinguished sign},
  author={Majda, A.},
  journal={Indiana University Mathematics Journal},
  pages={921--939},
  year={1993},
  publisher={JSTOR}
}

@article{evans1994hardy,
  title={Hardy spaces and the two-dimensional {E}uler equations with nonnegative vorticity},
  author={Evans, L. C. and M{\"u}ller, S.},
  journal={Journal of the American Mathematical Society},
  volume={7},
  number={1},
  pages={199--219},
  year={1994}
}

@book {Evans_notes,
    AUTHOR = {Evans, L. C.},
     TITLE = {Weak convergence methods for nonlinear partial differential
              equations},
    SERIES = {CBMS Regional Conference Series in Mathematics},
    VOLUME = {74},
 PUBLISHER = {Conference Board of the Mathematical Sciences, Washington, DC;
              by the American Mathematical Society, Providence, RI},
      YEAR = {1990},
     PAGES = {viii+80},
      ISBN = {0-8218-0724-2},
   MRCLASS = {35Bxx (35A25 58E15)},
  MRNUMBER = {1034481},
MRREVIEWER = {Pierre-Louis\ Lions},
       DOI = {10.1090/cbms/074},
       URL = {https://doi.org/10.1090/cbms/074},
}

@article{delort1991existence,
  title={Existence de nappes de tourbillon en dimension deux},
  author={Delort, J.-M.},
  journal={Journal of the American Mathematical Society},
  volume={4},
  number={3},
  pages={553--586},
  year={1991}
}

@article {Lopes97,
    AUTHOR = {Nussenzveig Lopes, H. J.},
     TITLE = {A refined estimate of the size of concentration sets for
              $2${D} incompressible inviscid flow},
   JOURNAL = {Indiana Univ. Math. J.},
  FJOURNAL = {Indiana University Mathematics Journal},
    VOLUME = {46},
      YEAR = {1997},
    NUMBER = {1},
     PAGES = {165--182},
      ISSN = {0022-2518,1943-5258},
   MRCLASS = {35Q30 (46N20 76C99)},
  MRNUMBER = {1462801},
MRREVIEWER = {Rodolfo\ Salvi},
       DOI = {10.1512/iumj.1997.46.1334},
       URL = {https://doi.org/10.1512/iumj.1997.46.1334},
}

@book{RR,
    AUTHOR = {Robinson, J. C. and Rodrigo, J. L. and Sadowski, W.},
     TITLE = {The three-dimensional {N}avier-{S}tokes equations},
    SERIES = {Cambridge Studies in Advanced Mathematics},
    VOLUME = {157},
      NOTE = {Classical theory},
 PUBLISHER = {Cambridge University Press, Cambridge},
      YEAR = {2016},
     PAGES = {xiv+471},
      ISBN = {978-1-107-01966-9},
   MRCLASS = {76-02 (35Q30 76D05)},
  MRNUMBER = {3616490},
MRREVIEWER = {Jean C. Cortissoz},
       DOI = {10.1017/CBO9781139095143},
       URL = {https://doi.org/10.1017/CBO9781139095143},
}

@article {CLLS16,
    AUTHOR = {Cheskidov, A. and Filho, M. C. Lopes and Lopes, H. J.
              Nussenzveig and Shvydkoy, R.},
     TITLE = {Energy conservation in two-dimensional incompressible ideal
              fluids},
   JOURNAL = {Comm. Math. Phys.},
  FJOURNAL = {Communications in Mathematical Physics},
    VOLUME = {348},
      YEAR = {2016},
    NUMBER = {1},
     PAGES = {129--143},
      ISSN = {0010-3616,1432-0916},
   MRCLASS = {76B03},
  MRNUMBER = {3551263},
MRREVIEWER = {Franck\ Sueur},
       DOI = {10.1007/s00220-016-2730-8},
       URL = {https://doi.org/10.1007/s00220-016-2730-8},
}

@article {DM87,
    AUTHOR = {DiPerna, R. J. and Majda, A.},
     TITLE = {Oscillations and concentrations in weak solutions of the
              incompressible fluid equations},
   JOURNAL = {Comm. Math. Phys.},
  FJOURNAL = {Communications in Mathematical Physics},
    VOLUME = {108},
      YEAR = {1987},
    NUMBER = {4},
     PAGES = {667--689},
      ISSN = {0010-3616,1432-0916},
   MRCLASS = {35Q10 (76D05)},
  MRNUMBER = {877643},
MRREVIEWER = {G.\ A.\ Nariboli},
       URL = {http://projecteuclid.org/euclid.cmp/1104116630},
}

@article {DT23,
    AUTHOR = {De Rosa, L. and Tione, R.},
     TITLE = {Fine properties of symmetric and positive matrix fields with
              bounded divergence},
   JOURNAL = {Adv. Math.},
  FJOURNAL = {Advances in Mathematics},
    VOLUME = {427},
      YEAR = {2023},
     PAGES = {Paper No. 109130, 42},
      ISSN = {0001-8708,1090-2082},
   MRCLASS = {39B42 (15B48 39B62 49J10 49J45)},
  MRNUMBER = {4597341},
       DOI = {10.1016/j.aim.2023.109130},
       URL = {https://doi.org/10.1016/j.aim.2023.109130},
}

@article {K41,
    AUTHOR = {Kolmogorov, A.},
     TITLE = {The local structure of turbulence in incompressible viscous
              fluid for very large {R}eynold's numbers},
   JOURNAL = {C. R. (Doklady) Acad. Sci. URSS (N.S.)},
  FJOURNAL = {C. R. (Doklady) Acad. Sci. URSS (N.S.)},
    VOLUME = {30},
      YEAR = {1941},
     PAGES = {301--305},
   MRCLASS = {76.1X},
  MRNUMBER = {4146},
MRREVIEWER = {W.\ R.\ Sears},
}

@article {L34,
    AUTHOR = {Leray, J.},
     TITLE = {Sur le mouvement d'un liquide visqueux emplissant l'espace},
   JOURNAL = {Acta Math.},
  FJOURNAL = {Acta Mathematica},
    VOLUME = {63},
      YEAR = {1934},
    NUMBER = {1},
     PAGES = {193--248},
      ISSN = {0001-5962},
   MRCLASS = {DML},
  MRNUMBER = {1555394},
       DOI = {10.1007/BF02547354},
       URL = {https://doi.org/10.1007/BF02547354},
}

@book {Lions_book,
    AUTHOR = {Lions, P.-L.},
     TITLE = {Mathematical topics in fluid mechanics. {V}ol. 1},
    SERIES = {Oxford Lecture Series in Mathematics and its Applications},
    VOLUME = {3},
      NOTE = {Incompressible models,
              Oxford Science Publications},
 PUBLISHER = {The Clarendon Press, Oxford University Press, New York},
      YEAR = {1996},
     PAGES = {xiv+237},
      ISBN = {0-19-851487-5},
   MRCLASS = {76-02 (35Q30 35Q35 76D05)},
  MRNUMBER = {1422251},
MRREVIEWER = {Denis\ Serre},
}

@article {GR23,
    AUTHOR = {Giri, V. and Radu, R.-O.},
     TITLE = {The {O}nsager conjecture in 2{D}: a {N}ewton-{N}ash iteration},
   JOURNAL = {Invent. Math.},
  FJOURNAL = {Inventiones Mathematicae},
    VOLUME = {238},
      YEAR = {2024},
    NUMBER = {2},
     PAGES = {691--768},
      ISSN = {0020-9910,1432-1297},
   MRCLASS = {35 (76)},
  MRNUMBER = {4809443},
       DOI = {10.1007/s00222-024-01291-z},
       URL = {https://doi.org/10.1007/s00222-024-01291-z},
}

@article {DRIS24,
    AUTHOR = {De Rosa, L. and Isett, P.},
     TITLE = {Intermittency and lower dimensional dissipation in
              incompressible fluids},
   JOURNAL = {Arch. Ration. Mech. Anal.},
  FJOURNAL = {Archive for Rational Mechanics and Analysis},
    VOLUME = {248},
      YEAR = {2024},
    NUMBER = {1},
     PAGES = {Paper No. 11, 37},
      ISSN = {0003-9527,1432-0673},
   MRCLASS = {76B03 (35Q31)},
  MRNUMBER = {4694411},
MRREVIEWER = {Anna\ L.\ Mazzucato},
       DOI = {10.1007/s00205-023-01954-w},
       URL = {https://doi.org/10.1007/s00205-023-01954-w},
}

@article {DDI24,
    AUTHOR = {De Rosa, L. and Drivas, T. D. and Inversi, M.},
     TITLE = {On the support of anomalous dissipation measures},
   JOURNAL = {J. Math. Fluid Mech.},
  FJOURNAL = {Journal of Mathematical Fluid Mechanics},
    VOLUME = {26},
      YEAR = {2024},
    NUMBER = {4},
     PAGES = {Paper No. 56, 24},
      ISSN = {1422-6928,1422-6952},
   MRCLASS = {35Q31 (35D30 35L65 76D05 76F02)},
  MRNUMBER = {4789345},
       DOI = {10.1007/s00021-024-00894-z},
       URL = {https://doi.org/10.1007/s00021-024-00894-z},
}

@article {DR00,
    AUTHOR = {Duchon, J. and Robert, R.},
     TITLE = {Inertial energy dissipation for weak solutions of
              incompressible {E}uler and {N}avier-{S}tokes equations},
   JOURNAL = {Nonlinearity},
  FJOURNAL = {Nonlinearity},
    VOLUME = {13},
      YEAR = {2000},
    NUMBER = {1},
     PAGES = {249--255},
      ISSN = {0951-7715},
   MRCLASS = {76D05 (35Q30 76B99)},
  MRNUMBER = {1734632},
MRREVIEWER = {Emmanuel Grenier},
       DOI = {10.1088/0951-7715/13/1/312},
       URL = {https://doi.org/10.1088/0951-7715/13/1/312},
}

@Article{PLL1,
  author    = {Lions, P.-L.},
  journal   = {Revista Matem{\'{a}}tica Iberoamericana},
  title     = {{The Concentration-Compactness Principle in the Calculus of Variations. The limit case, Part 1}},
  year      = {1985},
  pages     = {145--201},
  doi       = {10.4171/rmi/6},
  publisher = {European Mathematical Society - {EMS} - Publishing House {GmbH}},
}

@book {Struwe,
    AUTHOR = {Struwe, M.},
     TITLE = {Variational methods},
    SERIES = {Series of Modern Surveys in Mathematics},
    VOLUME = {34},
   EDITION = {Fourth Edition},
      NOTE = {Applications to nonlinear partial differential equations and
              Hamiltonian systems},
 PUBLISHER = {Springer-Verlag, Berlin},
      YEAR = {2008},
     PAGES = {xx+302},
      ISBN = {978-3-540-74012-4},
   MRCLASS = {49-02 (34C25 35A15 35F20 37J45 47J30 49J10 58E05)},
  MRNUMBER = {2431434},
}

@Article{PLL2,
  author   = {Lions, P.-L.},
  journal  = {Revista Matem\'atica Iberoamericana},
  title    = {{The Concentration-Compactness Principle in the Calculus of Variations. The limit case, Part 2}},
  year     = {1985},
  number   = {2},
  pages    = {45-121},
  volume   = {1},
  language = {eng},
  url      = {http://eudml.org/doc/39321},
}

@article {LMP21,
    AUTHOR = {Lanthaler, S. and Mishra, S. and Par\'{e}s-Pulido, C.},
     TITLE = {On the conservation of energy in two-dimensional
              incompressible flows},
   JOURNAL = {Nonlinearity},
  FJOURNAL = {Nonlinearity},
    VOLUME = {34},
      YEAR = {2021},
    NUMBER = {2},
     PAGES = {1084--1135},
      ISSN = {0951-7715,1361-6544},
   MRCLASS = {35Q35 (35Q31 65C05 65M12 65M70 76D03)},
  MRNUMBER = {4228012},
       DOI = {10.1088/1361-6544/abb452},
       URL = {https://doi.org/10.1088/1361-6544/abb452},
}

@article{O49,
   author={Onsager, L.},
   title={Statistical hydrodynamics},
   journal={Nuovo Cimento (9)},
   volume={6},
   date={1949},
   number={Supplemento, 2 (Convegno Internazionale di Meccanica
   Statistica)},
   pages={279--287},
}

@article {GT88,
    AUTHOR = {Greengard, C. and Thomann, E.},
     TITLE = {On {D}i{P}erna-{M}ajda concentration sets for two-dimensional
              incompressible flow},
   JOURNAL = {Comm. Pure Appl. Math.},
  FJOURNAL = {Communications on Pure and Applied Mathematics},
    VOLUME = {41},
      YEAR = {1988},
    NUMBER = {3},
     PAGES = {295--303},
      ISSN = {0010-3640,1097-0312},
   MRCLASS = {35Q10 (76C99)},
  MRNUMBER = {929281},
MRREVIEWER = {Jean\ Leray},
       DOI = {10.1002/cpa.3160410303},
       URL = {https://doi.org/10.1002/cpa.3160410303},
}

@article {CET94,
    AUTHOR = {Constantin, P. and E, W. and Titi, E. S.},
     TITLE = {Onsager's conjecture on the energy conservation for solutions
              of {E}uler's equation},
   JOURNAL = {Comm. Math. Phys.},
  FJOURNAL = {Communications in Mathematical Physics},
    VOLUME = {165},
      YEAR = {1994},
    NUMBER = {1},
     PAGES = {207--209},
      ISSN = {0010-3616},
   MRCLASS = {76C99 (35Q30 76F99)},
  MRNUMBER = {1298949},
       URL = {http://projecteuclid.org/euclid.cmp/1104271041},
}

@article{diperna1987concentrations,
  title={Concentrations in regularizations for $2$-{D} incompressible flow},
  author={DiPerna, R. J. and Majda, A.},
  journal={Comm. Pure App. Math.},
  volume={40},
  number={3},
  pages={301--345},
  year={1987},
  publisher={Wiley Online Library}
}

@article{DDII25,
   author={De Rosa, L. and Drivas, T. D. and Inversi, M. and Isett, P.},
   title={Intermittency and Dissipation Regularity in Turbulence},
   note={Preprint available at \href{https://arxiv.org/abs/2502.10032}{arXiv:2502.10032}},
   date={2025},
}

@article {GSPS21,
    AUTHOR = {G\'omez-Serrano, J. and Park, J. and Shi, J.},
     TITLE = {Existence of non-trivial non-concentrated compactly supported
              stationary solutions of the 2{D} {E}uler equation with finite
              energy},
   JOURNAL = {Mem. Amer. Math. Soc.},
  FJOURNAL = {Memoirs of the American Mathematical Society},
    VOLUME = {311},
      YEAR = {2025},
    NUMBER = {1577},
     PAGES = {v+82},
      ISSN = {0065-9266,1947-6221},
      ISBN = {978-1-4704-7530-7; 978-1-4704-8396-8},
   MRCLASS = {35Q35 (35Q31 35R35 76Dxx)},
  MRNUMBER = {4935841},
MRREVIEWER = {Qin\ Zhao},
       DOI = {10.1090/memo/1577},
       URL = {https://doi.org/10.1090/memo/1577},
}

\end{document}